\documentclass{amsart} %leqno is the option to put formula numbers on the left side

\usepackage{mathtools}
\usepackage{bbm}
\usepackage{cjhebrew}
\usepackage{amsmath, amssymb}
\usepackage{amsthm} 
\usepackage{amsfonts}
\usepackage{amscd}\usepackage{amsxtra}
\usepackage{multirow}
\usepackage[all,cmtip]{xy}
\usepackage{listings}
\usepackage{color}

\definecolor{dkgreen}{rgb}{0,0.6,0}
\definecolor{gray}{rgb}{0.5,0.5,0.5}
\definecolor{mauve}{rgb}{0.58,0,0.82}

\newcommand{\red}{{\text{\rm red}}}
\newcommand{\auto}{{\text{\rm auto}}}

\newcommand{\sA}{\mathcal{A}}

\newcommand{\sG}{\mathcal{G}}
\newcommand{\sH}{\mathcal{H}}
\newcommand{\sI}{\mathcal{I}}

\newcommand{\sL}{\mathcal{L}}
\newcommand{\sM}{\mathcal{M}}
\newcommand{\sO}{\mathcal{O}}

\newcommand{\sX}{\mathcal{X}}
\newcommand{\sY}{\mathcal{Y}}

\newcommand{\bL}{\mathbb{L}}

\newcommand{\bN}{\mathbb{N}}
\newcommand{\bZ}{\mathbb{Z}}

\newcommand{\bC}{\mathbb{C}}
\newcommand{\bA}{\mathbb{A}}

\newcommand{\fm}{\mathfrak{m}}
\newcommand{\fn}{\mathfrak{n}}

\newcommand{\fl}{\mathfrak{l}}

\newcommand{\Spec}{\operatorname{Spec}}

\newcommand\Var[1]{\mathbf{Var}_{#1}}

\newcommand\Form[1]{\categ{Form}_{#1}}

\newcommand{\op}{{\text{\rm op}}}

\newcommand{\inj}{\hookrightarrow}

\newcommand{\colim}{\varinjlim}

\renewcommand{\op}{\operatorname}

\renewcommand \dim[1]{\mbox{dim}{#1}}

\newcommand\sch[1]{\categ{Sch}_{#1}}
\newcommand\set{\mathbf{Sets}}

\newcommand\spec[1]{\operatorname{Spec}(#1)}

\newcommand\fatpoints[1]{\categ {Fat}_{#1}}

\newcommand\categ[1]{\mathbbmss{#1}}

\newcommand\into{\hookrightarrow}

\newcommand\sieves[1]{\categ{Sieve}_{#1}}
\newcommand\fim[1]{\texttt{Im}(#1)}

\newcommand\func[1]{#1^\circ}
\newcommand \grot[1]{{\mathbf {Gr}(#1)}}

\newcommand \nat{\mathbb N}

\newcommand \onto{\twoheadrightarrow}

\newcommand \tensor{\otimes}
\newcommand \complet[1]{\widehat {#1}}

\catcode`\@=11
\def\opn#1#2{\def#1{\mathop{\kern0pt\fam0#2}\nolimits}} % to make operators
\def\underrightarrow{\mathpalette\underrightarrow@}
\def\underrightarrow@#1#2{\vtop{\ialign{$##$\cr
 \hfil#1#2\hfil\cr\noalign{\nointerlineskip}%
 #1{-}\mkern-6mu\cleaders\hbox{$#1\mkern-2mu{-}\mkern-2mu$}\hfill
 \mkern-6mu{\to}\cr}}} 

\def\underleftarrow{\mathpalette\underlefalwaystarrow@}
\def\underleftarrow@#1#2{\vtop{\ialign{$##$\cr
 \hfil#1#2\hfil\cr\noalign{\nointerlineskip}#1{\leftarrow}\mkern-6mu
 \cleaders\hbox{$#1\mkern-2mu{-}\mkern-2mu$}\hfill
 \mkern-6mu{-}\cr}}}
    \let\phi=\varphi
    \let\epsilon=\varepsilon  
\def\:{\colon}   
\let\oldtilde=\tilde
\def\tilde#1{\mathchoice{\widetilde{#1}}{\widetilde{#1}}%
{\indextil{#1}}{\oldtilde{#1}}}
\def\indextil#1{\lower2pt\hbox{$\textstyle{\oldtilde{\raise2pt%
\hbox{$\scriptstyle{#1}$}}}$}}
\def\pnt{{\raise1.1pt\hbox{$\textstyle.$}}}  
\let\amp@rs@nd@\relax
\newdimen\ex@
\newdimen\bigaw@
\newdimen\minaw@
\newdimen\minCDaw@  
\newif\ifCD@
\def\minCDarrowwidth#1{\minCDaw@#1}
\renewenvironment{CD}{\@CD}{\@endCD}
\def\@CD{\def\A##1A##2A{\llap{$\vcenter{\hbox
 {$\scriptstyle##1$}}$}\Big\uparrow\rlap%
{$\vcenter{\hbox{$\scriptstyle##2$}}$}&&}%
\def\V##1V##2V{\llap{$\vcenter{\hbox
 {$\scriptstyle##1$}}$}\Big\downarrow\rlap%
{$\vcenter{\hbox{$\scriptstyle##2$}}$}&&}%
\def\={&\hskip.5em\mathrel
 {\vbox{\hrule width\minCDaw@\vskip3\ex@\hrule width
 \minCDaw@}}\hskip.5em&}%
\def\verteq{\Big\Vert&&}%
\def\noarr{&&}%
\def\vspace##1{\noalign{\vskip##1\relax}}%
\relax\iffalse{\fi\let\amp@rs@nd@&\iffalse}\fi
 \CD@true\vcenter\bgroup\relax%   
\iffalse{\fi\let\\=\cr\iffalse}\fi\tabskip\z@skip\baselineskip20\ex@
 &\hfill$\m@th##$\hfill\cr}
\def\@endCD{\cr\egroup\egroup}
\def\>#1>#2>{\amp@rs@nd@\setbox\z@\hbox{$\scriptstyle
 \;{#1}\;\;$}\setbox\@ne\hbox{$\scriptstyle\;{#2}\;\;$}\setbox\tw@
 \hbox{$#2$}\ifCD@
 \global\bigaw@\minCDaw@\else\global\bigaw@\minaw@\fi
 \ifdim\wd\z@>\bigaw@\global\bigaw@\wd\z@\fi
 \ifdim\wd\@ne>\bigaw@\global\bigaw@\wd\@ne\fi
 \ifCD@\hskip.5em\fi
 \ifdim\wd\tw@>\z@
 \mathrel{\mathop{\hbox to\bigaw@{\rightarrowfill}}\limits^{#1}_{#2}}\else
 \mathrel{\mathop{\hbox to\bigaw@{\rightarrowfill}}\limits^{#1}}\fi
 \ifCD@\hskip.5em\fi\amp@rs@nd@}
\def\<#1<#2<{\amp@rs@nd@\setbox\z@\hbox{$\scriptstyle
 \;\;{#1}\;$}\setbox\@ne\hbox{$\scriptstyle\;\;{#2}\;$}\setbox\tw@
 \hbox{$#2$}\ifCD@
 \global\bigaw@\minCDaw@\else\global\bigaw@\minaw@\fi
 \ifdim\wd\z@>\bigaw@\global\bigaw@\wd\z@\fi
 \ifdim\wd\@ne>\bigaw@\global\bigaw@\wd\@ne\fi
 \ifCD@\hskip.5em\fi
 \ifdim\wd\tw@>\z@
 \mathrel{\mathop{\hbox to\bigaw@{\leftarrowfill}}\limits^{#1}_{#2}}\else
 \mathrel{\mathop{\hbox to\bigaw@{\leftarrowfill}}\limits^{#1}}\fi
 \ifCD@\hskip.5em\fi\amp@rs@nd@}
\def\@CDS{\def\A##1A##2A{\llap{$\vcenter{\hbox
 {$\scriptstyle##1$}}$}\Big\uparrow\rlap%
{$\vcenter{\hbox{$\scriptstyle##2$}}$}&}%
\def\V##1V##2V{\llap{$\vcenter{\hbox
 {$\scriptstyle##1$}}$}\Big\downarrow\rlap%
{$\vcenter{\hbox{$\scriptstyle##2$}}$}&}%
\def\={&\hskip.5em\mathrel
 {\vbox{\hrule width\minCDaw@\vskip3\ex@\hrule width
 \minCDaw@}}\hskip.5em&}
\def\verteq{\Big\Vert&}
\def\novarr{&}
\def\noharr{&&}
\def\SE##1E##2E{\slantedarrow(0,18)(4,-3){##1}{##2}&}
\def\SW##1W##2W{\slantedarrow(24,18)(-4,-3){##1}{##2}&}
\def\NE##1E##2E{\slantedarrow(0,0)(4,3){##1}{##2}&}
\def\NW##1W##2W{\slantedarrow(24,0)(-4,3){##1}{##2}&}
\def\slantedarrow(##1)(##2)##3##4{\thinlines\unitlength1pt%
\lower 6.5pt\hbox{\begin{picture}(24,18)%
\put(##1){\vector(##2){24}}%
\put(0,8){$\scriptstyle##3$}%
\put(20,8){$\scriptstyle##4$}%
\end{picture}}}
\def\vspace##1{\noalign{\vskip##1\relax}}\relax%
\iffalse{\fi\let\amp@rs@nd@&\iffalse}\fi
 \CD@true\vcenter\bgroup\relax\iffalse%
{\fi\let\\=\cr\iffalse}\fi\tabskip\z@skip\baselineskip20\ex@
 &\hfill$\m@th##$\hfill\cr}
\def\@endCDS{\cr\egroup\egroup}
\theoremstyle{plain} %text of this environment is typesetted in italics
\newtheorem{theorem}{\indent\sc Theorem}[section]
\newtheorem{lemma}[theorem]{\indent\sc Lemma}
\newtheorem{corollary}[theorem]{\indent\sc Corollary}
\newtheorem{proposition}[theorem]{\indent\sc Proposition}

\newtheorem{conjecture}[theorem]{\indent\sc Conjecture}
\theoremstyle{definition} %text of this environment is typesetted in roman letters
\newtheorem{definition}[theorem]{\indent\sc Definition}
\newtheorem{remark}[theorem]{\indent\sc Remark}
\newtheorem{example}[theorem]{\indent\sc Example}

\newtheorem{question}[theorem]{\indent\sc Question}

\makeatletter
\def\address#1#2{\begingroup
\noindent\parbox[t]{7.8cm}{%
\small{\scshape\ignorespaces#1}\par\vskip1ex
\noindent\small{\itshape astout@gc.cuny.edu}%
\/: #2\par\vskip4ex}\hfill%
\endgroup}%
\makeatother
\begin{document}

\title{On the Auto Igusa-Zeta function of an algebraic curve.}

\author{Andrew R. Stout}

\email{astout@gc.cuny.edu}

\maketitle
%ABSTRACT

\begin{abstract}
 We study endomorphisms of complete Noetherian local rings in the context of motivic integration. Using the notion of an auto-arc space, we introduce the (reduced) auto-Igusa zeta series at a point, which appears to measure the degree to which a variety is not smooth that point. We conjecture a closed formula in the case of curves with one singular point, and we provide explicit formulas for this series in the case of the cusp and the node. Using the work of Denef and Loeser, one can show that this series will often be rational. These ideas were obtained through extensive calculations in Sage. Thus, we include a Sage script which was used in these calculations. It computes the affine arc spaces $\nabla_{\fn}X$ provided that $X$ is affine, $\fn$ is a fat point, and the ground field is of characteristic zero. Finally, we show that the auto Poincar\'{e} series will often be rational as well and connect this to questions concerning new types of motivic integrals.
\end{abstract}
\setcounter{tocdepth}{1} \tableofcontents{}

\section*{Introduction}

We study endomorphisms of complete Noetherian local rings 
and their connection with motivic integration. More clearly, 
let $(R,\fm)$ be a complete Noetherian  local ring with residue 
field $k$, then we study the sequence of schemes $\sA_n$ which 
represent the functor $\fatpoints{k}\to \set$ defined by
\begin{equation*} S\mapsto\mbox{Mor}_k(\spec{R/\fm^n}\times_k S, 
\spec{R/\fm^n})
\end{equation*}
where $\fatpoints{k}$ is the full subcategory of separated 
schemes of finite type over $k$, denoted here by $\sch{k}$, 
whose objects are connected and zero dimensional. Almost 
always, the schemes $\sA_n$ are highly non-reduced for $n> 2$. 
Thus, it is natural to consider their reduction 
$\sA_n^{\red}$.  In \S 2, we describe how we use reduction as 
a ring homomorphism of two Grothendieck rings -- i.e., the 
Grothendieck ring of the formal site $\grot{\Form{k}}$ and the 
Grothendieck ring of varieties $\grot{\Var{k}}$. In \S 3, we formally introduce the object $\sA_n$, and we name them {\it auto-arc spaces}. Likewise, in \S 4, we introduce the auto Igusa-zeta series, which is a generating series for the sequence $\sA_n$, with appropriately normalization via  negative powers of the leftschetz motive $\bL$. Notably, in \S 3 and \S 4, we make conjectures concerning the degree to which auto-arc spaces measure smoothness (resp. \'{e}taleness). For example, we conjecture that if $\bar \zeta_{X,p}(t)=\bL^{-\dim{}_p(X)}\frac{1}{1-t}$, then $X$ is smooth at $p$.

In Sections 5, 6, and 7, we investigate this conjecture in the case where $X$ is an algebraic curve. As the spaces $\sA_{n}(X,p)$ are generally quite impossible to compute by hand when $n$ is large, we implement a Sage script the author coded which computes $\sA_n(X,p)$. We carry out this computation in \S 5 and begin to notice some patterns in the case of curves. For example, we notice that if $C$ is the cuspidal cubic (given by $y^2=x^3$) and $O$ is the origin, then for $n=4,5, 6$, then 
\begin{equation*}
\sA_n(C,O)^{\red}\cong \nabla_{\fl_{2(n-3)}}C\times_k\bA_{k}^{n} .
\end{equation*}
In \S 6, we prove this formula is valid for all $n\geq 4$. Moreover, we prove a similar formula in the case of the node $N$. This leads us to make a conjecture about the structure of $\sA_n(X,p)^{\red}$ when $X$ is an algebraic curve with only one singular point $p$. Naturally, this leads us in Section 7 to investigate the auto Igusa-zeta series in the case of algebraic curves with only one singular point. We show in the case of smooth curves, the cuspidal cubic, the node, and the nodal cubic, that the auto Igusa-zeta series  is intimately connected with the motivic Igusa-zeta series along the linear arc $\fl$ in this case (which is studied in \cite{DL2} and further generalized in \cite{Sch2}). In fact, we explicitly calculate the auto-Igusa zeta function in this case to obtain:
\begin{equation*}
\begin{split}
(X,p) &=\mbox{(cuspidal cubic, origin)} \implies \bar\zeta_{X,p}(t) = \frac{1-(\bL+1)t^3+\bL t^4 +(\bL-1)t^5 +2\bL^2t^6}{(1-\bL t^3)(1-t)} \\
(X,p) &=\mbox{(node, origin)} \implies \bar\zeta_{X,p}(t) = \frac{1 - (\bL^2 +4\bL -3)t+\bL^2(2\bL^2-1)t^2 - \bL^4(3\bL^2-1)t^3}{(1-\bL^2t)^3}
\end{split}
\end{equation*}
This leads us to make further conjectures about the structure of the auto Igusa-zeta series in the case of algebraic curves which have only one singular point. Whether or not this  structural conjecture is true, what is clear is that $\sA_n(X,p)^{\red}$ will be a semi-algebraic subset the traditional arc space $\sL(W)$ for some algebraic variety $W$. More work should be carried out in the future to investigate the case of mild singularities of higher dimensional varieties.

In \S 8, we explore the connection between the potential rationality of the auto Igusa-zeta series and new types of motivic volumes via auto-arc spaces. With a simple adjustment we show that there are such geometric motivic volumes by using the notion of the auto Poincar\'{e} series. Here, just as in the previous paragraph, we are making use of the celebrated theorem of J. Denef and F. Loeser, cf. \cite{DL1}, concerning the rationality of motivic Poincar\'{e} series and geometric motivic integration. Finally, we use  Appendix A, B, and C  to tie up any lose ends and to provide the code used in the computations which occur in \S 5. Note that the code may be used to calculate any arc space $\nabla_{\fn}X$ provided that $X$ is an affine scheme, $\fn$ is a fat point, and the ground field has characteristic zero. However, the computation speed is destined to be quite slow when the length of $\fn$ is large or when $X$ is complicated.

\section{Background}\label{finite}

We now quickly give an introduction to Schoutens' theory of schemic Grothendieck rings. Much of what is stated here is taken directly from \cite{Sch1} and \cite{Sch2}. Let $\sch{k}$ be the category of separated schemes of finite type over a field $k$. We form the Grothendieck pre-topology $\categ{Form}_{k}$ on $ \sch{k}$ in the following way.

\begin{definition}\label{morofsieves}
Given two sieves $\sX$ and $\sY$, we say
that a natural transformation $\nu : \sY \to \sX$ is {\it a morphism of
 sieves} if
given any morphism of schemes $\varphi : Z \to Y$ such that
$\fim{\func{\varphi}} \subset \sY$,
there exists a morphism of schemes $\psi : Z \to X$ with $\sX
\subset X$ such that the following diagram commutes 
\[\xymatrix{\entrymodifiers={+!!<0pt,\fontdimen22\textfont2>}&Z^{\circ}  \ar[d]_{\func{\varphi}}  \ar[rrd]^{\func{\psi}}\\&\sY \ar[r]_{\nu} &\sX \ar@{^{(}->}[r]_{\iota}&X^{\circ} } \]
where $\iota$ is the natural inclusion defining $\sX$ as a subfunctor of $X^{\circ} := \mbox{Mor}_{\sch{k}}(-,X)$.
This forms a category which we denote by $\sieves{k}$.
\end{definition}

We say that a sieve $\sY$ is {\it subschemic} if it is of the form $\fim{\varphi^{\circ}}$ where $\varphi : X \to Y$ is a morphism in $\sch{k}$. The collection of subschemic sieves satisfies the axioms of a Grothendieck pre-topology; however, it is not all that interesting as we have the following theorem due to Schoutens:

\begin{theorem}
Let $\nu : \sY \to \sX$ be a continuous morphism in $\sieves{k}$ and assume that $\sX$ and $\sY$ are subschemic and that $\sY$ is affine. Then, $\nu$ is rational -- i.e., there exists a morphism $\varphi:Y \to X$ in $\sch{k}$ such that 
\begin{equation*}
\varphi^{\circ} \circ\iota = \nu \ ,
 \end{equation*}
where $\iota: \sY \into Y^{\circ}$ is a natural inclusion.
\end{theorem}

\begin{proof}
This is a restatement of Theorem 3.17 of \cite{Sch1}. A proof may be found there.
\end{proof}

However, there is a large class of sieves which do not have this property. Recall the construction of a formal scheme. One starts with a closed subscheme $Y$ of $X$ with corresponding ideal sheaf $\sI_Y$. For each $n \in \nat$, $\sI_{Y}^{n}$ is a quasi-coherent sheaf of ideals of $\sO_{X}$. Thus, we have the closed subscheme $Y_n$ of $X$ determined by the ideal sheaf $\sI_{Y}^{n}$. Then, the formal scheme of $X$ along $Y$ is the  locally ringed topological space $\complet Y$ which is isomorphic to the filtered colimit $\colim_{n\in\nat} Y_n$. This leads us to make the following definition.
\begin{definition}\label{form}
We say that a sieve $\sX$ is {\it formal} if for each connected finite ${k}$-scheme $\fm$, there is a subschemic sieve $\sY_{\fm}\subset\sX$ such that the sets
$\sY_{\fm}(\fm)$ and $\sX(\fm)$ are equal.
\end{definition}
In Theorem $7.8$ of \cite{Sch1}, Schoutens proved that the collection of all formal sieves, denoted by $\categ{Form}_{k}$ is a Grothendieck pre-topology. It can be shown as well that categorical product and coproduct commute in the full subcategory  $\categ{Form}_{k}$ of $\sieves{k}.$ Thus, using at the beginning of \S 4.1 of \cite{Sch1}, we may form the Grothendieck ring of  $\categ{Form}_{k}$ . We denote the resulting ring by  $\grot{\categ{Form}_{k}}$ and call it {\it the Grothendieck ring of the formal site}. By Proposition \ref{surj}, there is a surjective ring homomorphism
\begin{equation}
\grot{\categ{Form}_{k}} \onto \grot {\mathbf{Var}_{k}} \ .
\end{equation}

In motivic integration, one often deals with the arc space\footnote{In this paper, we use the notation of $\nabla_{\fl_n}X$ in place of $\sL_{n+1}(X)$ to denote the truncated arc space. Likewise, we will let $\nabla_{\fl}X$ (and not $\sL(X)$) denote the infinite arc space of $X$.} $\mathcal{L}(X)$ which is the projective limit of the  $n$-th order arc spaces. The truncated arc space $\sL_{n}(X)$ is defined to be the separated scheme of finite type over $k$ representing the functor from connected $k$-schemes which are finite over $k$ to $\set$:
\begin{equation*}
\fm \mapsto X^{\circ}(\fm \times_{k}\spec{k[t]/(t^{n+1})}) \ .
\end{equation*}
Usually, one only considers the reduced structure on $\sL(X)$.

Let $\fatpoints{k}$ be the full subcategory of $\sch{k}$ whose objects are connected finite $k$-schemes. We call $\fm \in\fatpoints{k}$ a {\it fat point} over $k$. All sieves $\sX$ restrict to $\fatpoints{k}$. We will abuse notation and denote the restriction of a sieve $\sX$ to $\fatpoints{k}$ as $\sX$ as well. Moreover, we will also denote the resulting category of all sieves $\sX$ restricted to $\fatpoints{k}$ by $\sieves{k}$.
The reason that we may perform this restriction is due to the following fact.

\begin{theorem}
Let $X$ and $Y$ be closed subschemes contained in a separated $k$-scheme $Z$ of finite type over $k$. Then, $X$ and $Y$ are non-isomorphic over $k$ if and only if there exists $\fm\in\fatpoints{k}$ such that  $X^{\circ}(\fm)$ and $Y^{\circ}(\fm)$
are distinct subsets of $Z^{\circ}(\fm)$.
\end{theorem}

\begin{proof}
This is a restatement of Lemma 2.2 of \cite{Sch1}. A proof can be found there.
\end{proof}

One of the goals of Schoutens' work is to show that the construction of the arc space works when we replace $\spec{k[t]/(t^{n})}$ with an arbitrary fat point $\fn$ in the context of motivic integration. This leads us to define the generalized arc space of a sieve $\sX$ along the fat point $\fn$ by
\begin{equation}\label{arcop}
\nabla_{\fn}\sX(-) := \sX(-\times_{k}\fn)
\end{equation}
as a functor from $\fatpoints{k}$ to $\set$. Schoutens proved in \S $3$ of \cite{Sch2} that if $\sX = X^{\circ}$ for some $X\in\sch{k}$, then it follows that $\nabla_{\fn}\sX$ is a represented by an element of $\sch{k}$. Thus, it follows immediately that $\nabla_{\fn}\sX \in \sieves{k}$ for any $\sX \in \sieves{k}$ and any $\fn\in\fatpoints{k}$. Moreover, Schoutens showed that if $\sX$ is formal, then so is $\nabla_{\fn}\sX$. Of course, when $\sX$ is merely a scheme, then the generalised arc space is similar to Weil restriction, which is studied and partially generalised by many authors\footnote{What is unique in Schoutens' approach, ignoring the context to motivic integration and formal sieves, is realising $\nabla_{\fn}$ as a composition $\nabla_{j*} \circ \nabla_{j_*}$, cf. \S 3 and \S 4 of \cite{Sch2}. This approach proves useful in the proceeding section.}.

\section{Some remarks on the Grothendieck ring of the formal site}\label{group}

In this section, we recall how to complete the Grothendieck ring of varieties over a field $k$, and then we show how this construction extends to $\grot{\Form{k}}$. This is necessary for two reasons. First, we will see that our motivic generating functions have strict counterparts with coefficients in $\grot{\Form{k}}[\bL^{-1}]$, thus it becomes interesting to ask when motivic rationality occurs over $\grot{\Form{k}}[\bL^{-1}]$ and not just $\grot{\Var{k}}$, and, secondly, the sieve approach appears necessary when defining the reduced infinite auto-arc space along a germ in \S 8 of this paper, which means that there could be the possibility of defining, at least in some cases, the non-reduced infinite auto-arc space along a germ for formal sieves which would yield a motivic integral with values the completion of $\grot{\Form{k}}[\bL^{-1}]$. As we are not currently aware of how to extend the definition of an infinite auto-arc space along a germ when the underlying scheme is singular, it is believed that this is a crucial ingredient to the theory. 

Now, we let $\Var{k}$ denote the full subcategory of $\sch{k}$ formed by objects $X$ such that $X^{\red} \cong X$ and we call such an object a {\it variety over} $k$. In other words, a variety is a reduced separated scheme of finite type over $k$. We may construct the Grothendieck ring $\grot{\Var{k}}$ by forming the free abelian group on isomorphism classes of objects of $\Var{k}$ and imposing the so-called {\it scissor relations} 
\begin{equation*}
\langle X\cup Y\rangle -  \langle X\rangle - \langle Y\rangle + \langle X\cap Y\rangle
\end{equation*}
whenever $X, Y \in \Var{k}$ are locally closed subvarieties of some variety $V$ and where the brackets $\langle \cdot \rangle$ denote isomorphism classes. We then have a universal additive invariant 
\begin{equation*}
[\cdot] :\mathbf{Var}_{k} \to \grot{\mathbf{Var}_{k}} .
\end{equation*}
 The fiber product between two $k$-varieties over $k$ induces the structure of a commutative ring with a unit (this unit is $[\spec{k}] = 1$). There is another distinguished element of $\grot{\Var{k}}$ which is the so-called {\it Leftschetz motive} $\bL :=[\bA_{k}^{1}]$. We may invert this element to obtain a ring $\sG_{k} : = \grot{\Var{k}}[\bL^{-1}]$, and, moreover, the set-theoretic function $\dim : \Var{k} \to \bZ\cup\{-\infty\}$ which sends a variety to its dimension\footnote{Note that we formally define the dimension of the empty variety to be $-\infty$.} induces set-theoretic functions 
 \begin{equation}\label{dim}
 \begin{split}
 \dim \ :  \ \grot{\Var{k}}\to \bZ\cup\{-\infty\}\\
 \dim \ :\ \sG_k\to \bZ\cup\{-\infty\},
 \end{split}
 \end{equation} where the first function is given by $\dim(\sum_{i=1}^{m}[X_i]) = \max_{i=1\ldots m}\{\dim(X_i)\}$ and where the second function is given by $\dim(\frac{S}{1}\bL^{-i}) = \dim(S) - i$
 for each $i\in\bN$ and for each $S \in \grot{\Var{k}}$ where $\frac{S}{1}$ denotes the image of $S$ in the localization $\sG_k$. Then, we have the so-called {\it dimensional filtration on} $\sG_k$, which, for $m, n\in \bN$, looks like the following:
 \begin{equation*}
 \begin{split}
 {0}\subset\cdots\subset F^{-m-1}\sG_k\subset F^{-m}\sG_k\subset F^{-m+1}\sG_k \subset\cdots\\
 \cdots\subset F^{-1}\sG_k\subset F^0\sG_k\subset F^1\sG_k \subset \cdots\\
 \cdots\subset F^{n-1}\sG_k \subset F^{n}\sG_k \subset F^{n+1}\sG_k\subset \cdots \subset \sG_k,
 \end{split}
 \end{equation*}
where  $\sG_k = \cup_{i\in\bZ}F^{i}\sG_k$. More explicitly, for each $i\in\bZ$, one defines the above subgroups of $\sG_k$ by $F^i\sG_k:=\{T\in\sG_k \mid \dim(T) < i\}$. From this, we may consider the projective system of factor groups $\{\sG_k/F^i\sG_k \mid i \in \bZ\}$ where the group homomorphism 
\begin{equation*}
\sG_k/F^{i-1}\sG_k \to \sG_k/F^i\sG_k
\end{equation*}
is induced by modding out  $\sG_k/F^{i-1}\sG_k$ by the image of $F^i\sG_k$ in $\sG_k/F^{i-1}\sG_k$. Thus, we may form the projective limit
\begin{equation*} 
\hat\sG_k := \varprojlim_{i} \sG_k/F^i\sG_k 
\end{equation*}
to obtain a complete, topological ring (the multiplication in $\sG_k$ will extend to $\hat\sG_k$).

Now, we will follow the same procedure above to construct rings $\sH_k$ and $\hat\sH_k$ from $\grot{\Form{k}}$. First, we need the following theorem.

\begin{theorem}\label{surj}
For each $\fn\in\fatpoints{k}$, there is a surjective ring homomorphism
\begin{equation*}
\sigma_{\fn} : \grot{\Form{k}}\to \grot{\Var{k}} .
\end{equation*}
\end{theorem}

\begin{proof}
Fix a fat point $\fn\in\fatpoints{k}$ and let $\sX\in \Form{k}$. 
By Definition \ref{form}, there exists a subschemic sieve $\sY_{\fn}\subset \sX$ such that $\sY_{\fn}(\fn)=\sX(\fn)$. Furthermore, since $\sY_{\fn}$ is subschemic, there exists a morphism $\phi_{\fn} : Z \to Y$ in $\sch{k}$ such that the induced morphism of sieves $\phi^{\circ} : X^{\circ} \to Y^{\circ}$ is such that $\fim{\phi^{\circ}} = \sY_{\fn}$. By Theorem 1.8.4 of \cite{G2}, the image $C_{\fn}$ of $\phi_{\fn}$ is a constructible subset of $Y$. Moreover, by 1.8.2 of loc. cit., the inverse image $C_{\fn}^{\red}$ along the reduction morphism $Y^{\red}\to Y$ is a constructible subset of $Y^{\red}$. Thus, we define a set-theoretic function
\begin{equation}
\sigma_{\fn}:\grot{\Form{k}}\to \grot{\Var{k}}, \quad [\sX] \mapsto [C_{\fn}^{\red}] .
\end{equation}
The fact that $\sigma_{\fn}$ is well-defined and a surjective ring homomorphism follows mutatis mutandis from the argument in the proof of Theorem 7.7 of \cite{Sch1}. The only change that needs to be made to that argument is to apply the Theorem 1.8.4 of \cite{G2} in order to get rid of the assumption that the ground field is algebraically closed and to replace the fat point defined by the spectrum of the ground field by any arbitrary fat point $\fn\in\fatpoints{k}$.
\end{proof}

There are two remarks to be made concerning the above proof. First a notational issue arises because we use square brackets $[\cdot]$ to denote elements of $\grot{\Form{k}}$ and to 
denote elements of $\grot{\Var{k}}$. However, in context, it will always either be clear or, otherwise, explicitly stated to which Grothendieck ring the class of some object belongs. As we are dealing also with localizations and completions of Grothendieck rings, we find it best to avoid, say, subscripts to the square brackets as it would increase notation substantially. Likewise, in any Grothendieck ring, we will denote the class $[\bA_{k}^{1}]$ of the affine line by $\bL$ and call it {\it the Leftschetz motive}. Again, this should not lead to confusion on the part of the reader as it will be clear from the context in which Grothendieck ring $\bL$ dwells.

The second remark to be made is that one should note that the map $\sigma_{\fn}$ does indeed coincide with the surjective ring homomorphism $\grot{\Form{k}}\to \grot{\Var{k}}$ introduced in \cite{Sch1} when $\fn= \spec{k}$ and $k$ is algebraically closed. This is because the notions ``definable'' and ``constructible'' coincide over an algebraically closed field, cf. Corollary 3.2(i) of \cite{M}. Further, for notational convenience, we will sometimes write $\sigma_{R}$ for $\sigma_{\fn}$ whenever $R$ is the coordinate ring of $\fn$; this is particularly the case when $R$ is a field.

\begin{proposition}
Let $\sX$ be any object of $\Form{k}$ and let $\fn$ and $\fm$ be objects of $\fatpoints{k}$. Then, 
\begin{equation*}
\sigma_{\fn\times_k\fm}([\sX]) = \sigma_{\fm}([\nabla_{\fn}\sX]).
\end{equation*}
In particular, when $\fm=\spec{k}$, then 
\begin{equation*}
\sigma_{\fn}([\sX]) = \sigma_{k}([\nabla_{\fn}\sX]).
\end{equation*}
\end{proposition}

\begin{proof}
First, we will prove the second claim. 
The second claim follows from the fact that the functors $\nabla_{\fn}\sX$ and $\sX(\fn\times_k-)$ are adjoint. Indeed, the $\fn$-points of $\sX$ will be such that there is a subschemic sieve $\sY_{\fn}$ with the property that 
\begin{equation*}
\sY_{\fn}(\fn) =  \sX(\fn) = \nabla_{\fn}\sX(k) =\nabla_{\fn}\sY_{\fn}(k).
\end{equation*}
Thus, using adjunction again, the morphism of schemes $\phi : Z_1 \to Y_1$ determining $\sY_{\fn}$ and the morphism of schemes $\psi : Z_2 \to Y_2$ determining $\nabla_{\fn}\sY_{\fn}$ will be such that $\nabla_{\fn}\phi = \psi$. In other words, $[\nabla_{\fn}\sX]$ and $[\sX]$ are both sent to the class of the same constructible subset  $C^{\red}$ of $(\nabla_{\fn}Y_1)^{\red} = Y_2^{\red}$ under $\sigma_{k}$ and $\sigma_{\fn}$, respectively.

The proof for arbitrary $\fm\in\fatpoints{k}$ follows either mutatis mutandis by simply replacing $\spec{k}$ with $\fm$ in the above argument, or, by using the second claim twice to arrive at
$$\sigma_{\fn\times_k\fm}([\sX]) = \sigma_{k}([\nabla_{\fn\times_k\fm}\sX]) = \sigma_{\fm}([\nabla_{\fn}\sX])$$
as $\nabla_{\fn\times_k\fm}\sX \cong \nabla_{\fm}(\nabla_{\fn}\sX)$.
\end{proof}

\begin{remark} In \cite{Sch2}, $\nabla_{\fn}$ is often treated as a ring endomorphism of $\grot{\Form{k}}$. Thus, in that notation, one could  write $\sigma_{\fn\times_k\fm} = \sigma_{\fm}\circ \nabla_{\fn}$. In particular, we have shown that for all $\fn\in\fatpoints{k}$ the ring homomorphism $\sigma_k$ induces a surjective ring homomorphism from $\fim{\nabla_{\fn}}$ to $\grot{\Var{k}}$. 
\end{remark}

For each $\fn\in\fatpoints{k}$, we let $S_{\fn}$ denote the 
set $$\{s\bL^i\in\grot{\Form{k}} \mid \sigma_{\fn}(s)=1, \ i\in\bN\}.$$ Clearly, 
$S_{\fn}$ is stable under multiplication, and thus, we may 
localize $\grot{\Form{k}}$ by $S_{\fn}$ to obtain a ring 
$S_{\fn}^{-1}\grot{\Form{k}}$. Moreover, as we noted earlier 
in the proof of Theorem \ref{surj}, to each element of 
$\sX \in \Form{k}$, 
we may assign a constructible subset $C_{\fn}^{\red}$ 
of some variety, and this assignment defines the ring 
homomorphism $\sigma_{\fn}$. Thus, for each $\fn \in 
\fatpoints{k}$, we may define a set-theoretic function from 
$\Form{k}$ to $\bZ\cup\{-\infty\}$ by sending  
$\sX$ to $\dim(C_{\fn}^{\red})$ and this also defines the set-theoretic  
$\grot{\Form{k}}$ to $\bZ\cup\{-\infty\}$ defined by $[\sX]\mapsto\dim(\sigma_{\fn}([\sX])$. Clearly then, this 
will extend to a set-theoretic function  from 
$S_{\fn}^{-1}\grot{\Form{k}}$ to $\bZ\cup\{-\infty\}$ for each 
$\fn\in\fatpoints{k}$. Therefore, we have a filtration 
$\{F^i\sH\mid i\in\bZ\}$ of $S_{\fn}^{-1}\grot{\Form{k}}$ by 
subgroups defined by 
\begin{equation*}
F^i\sH := \{ T \in S_{\fn}^{-1}\grot{\Form{k}}\mid \dim(\sigma_{\fn}^{\prime}(T)) < i\}
\end{equation*}
where $\sigma_{\fn}^{\prime}$ is the induced ring homomorphism from $S_{\fn}^{-1}\grot{\Form{k}}$ to $\sG_k$ and $\dim\ $ is the function defined in Equation \ref{dim}. Thus, we may from the  group completion 
\begin{equation*}
S_{\fn}^{-1}\grot{\Form{k}}\hat\ := \varprojlim_{i\in\bZ}S_{\fn}^{-1}\grot{\Form{k}}/F^i\sH.
\end{equation*}
Multiplication in $S_{\fn}^{-1}\grot{\Form{k}}$ will extend to 
$S_{\fn}^{-1}\grot{\Form{k}}\hat\ $, which gives this complete topological group the structure of a complete topological ring.

In exactly the same way, we form the rings 
\begin{equation}
\begin{split}
\sH_{k}&:=\grot{\Form{k}}[\bL^{-1}]\\
\hat\sH_{k} &:=\varprojlim_{i\in\bZ} \sH_k/F^i\sH_k
\end{split}
\end{equation}
where $F^i\sH_k = \{T\in\sH_k\mid \dim(\sigma_{k}^{\prime}(T)) < i\}$.
Note that
$$\sH_k \subset\bigcap_{\fn\in\fatpoints{k}}S_{\fn}^{-1}\grot{\Form{k}}.$$
In particular, $\sH_k \subset S_{\spec{k}}^{-1}\grot{\Form{k}}$, and, for example, this clearly implies that 
$\hat\sH_k \subset S_{\spec{k}}^{-1}\grot{\Form{k}}\hat\ $. Thus, we have the following immediate lemma.

\begin{lemma}
For each $\fn\in\fatpoints{k}$, $\sigma_{\fn}$ induces canonical ring homomorphisms
\begin{equation*}
\begin{split}
\sigma_{\fn}^{\prime} &: S_{\fn}^{-1}\grot{\Form{k}} \to \sG_k\\
\hat\sigma_{\fn} &:S_{\fn}^{-1}\grot{\Form{k}}\hat\ \to \hat\sG_k
\end{split}
\end{equation*}
which are surjective and continuous morphisms of topological rings. Furthermore, $\sigma_{\fn}^{\prime}$ (resp. $\hat\sigma_{\fn}$) restrict to the canonical ring homomorphism $\sH_k \to \sG_k$ (resp. $\hat\sH_k\to \hat\sG_k$) induced by $\sigma_{\fn}$ and, these ring homomorphisms are also surjective and continuous morphisms of topological rings.
\end{lemma}

\begin{conjecture}
Let $\sX\in\Form{k}$ be such that $[\sX]=\bL^n$ in $\grot{\Form{k}}$. Then, there is a $\Form{k}$-homeomorphism  $f:\sX \to (\bA_{k}^n)^{\circ}$. In particular, $\sX = (\bA_{k}^n)^{\circ}$.
\end{conjecture}

\begin{remark}
Note that the analogue of this statement for $\grot{\Var{k}}$ is not true. Example 7.12 of \cite{Sch1} shows that $[C]=\bL$ in $\grot{\Var{k}}$ when $C$ is the nodal cubic. However, in $\grot{\Form{k}}$,  I doubt that this equality holds. For a more detailed study of this phenomenon see \cite{LS}.
\end{remark}

\section{Auto-arc spaces}\label{auto}

\begin{definition}
Let $X$ be an object of $\sch{k}$ and let $p$ be a point of $X$. Then, for each $n\in \bN$, we let $J_p^nX$ denote the scheme $\spec{\sO_{X,p}/\mathfrak{m}_p^n}$
and call it  {\it the $n$-jet of} $X$ {\it at the point} $p$. We always consider it as an object in the category $\sch{\kappa(p)}$.
\end{definition}

\begin{remark}
For $n=1$ and any object $X$ of $\sch{k}$, $J_p^nX = \spec{\kappa(p)}$ for any point $p$ of $X$. Moreover, for any object $X$ of $\sch{k}$ and any point $p$ of $X$, we have
\begin{equation*}
\colim_n J_{p}^{n}X  \cong (\spec{\kappa(p}), \hat\sO_{X,p})
\end{equation*}
in the category of locally ringed spaces. Thus, a filtered colimit of $n$-jets of $X$ at $p$ is just the one-point formal scheme defined by the completion of $\sO_{X,p}$ along  $\fm_p$. 
\end{remark}

\begin{definition}
Let $X$ be an object of $\sch{k}$ and let $p$ be a point of $X$. For each $n\in \bN$, we define {\it the auto-arc space of} $X$ {\it at} $p$ {\it of order} $n$ to be
\begin{equation}
\sA_n(X,p) := \nabla_{J_p^nX}J_p^nX .
\end{equation} 
\end{definition}

\begin{remark}
Clearly, $\sA_n(\spec{F},(0)) \cong \spec{F}$ for any field extension $F$ of $k$. In other words, $J_p^nX$ is always considered as a object of $\sch{\kappa(p)}$ and the functor $\nabla_{J_p^nX}(-)$ is always defined to be a functor from $\sch{\kappa(p)}$ to $\sch{\kappa(p)}$. Thus, in this exposition, one will not lose much by assuming $k$ is algebraically closed and $p$ is a closed point. 
\end{remark}

\begin{example}\label{linauto}
Consider the case where $X$ is $\bA_{\kappa}^{1}$ and let $p$ be any point of $\bA_{k}^{1}$. Then, we let $\fl_{\kappa(p),n}$ denote the jet space $J_p^nX$ where $\kappa(p)$ is the residue field of $p$. In other words,
\begin{equation*}
\fl_{\kappa(p),n}:=\spec{\kappa(p)[t]/(t^n)} \ .
\end{equation*}
We will calculate the reduction of the auto-arc space $\sA_n(\bA_{\kappa}^1,p)$. To do this, we let $\alpha:= \sum_{i=0}^{n-1} a_it^i \in \kappa(p)[t]/(t^n)$ with $a_i \in \kappa(p)$ and set $\alpha^n=0$ as an element of $\kappa(p)[t]/(t^n)$. Now, $0=\alpha^n = (a_0 + t\cdot\beta)^n$ where $\beta \in \kappa(p)[t]/(t^n)$, which implies that $a_0^n=0$ and so $a_0= 0$ in the reduction. Therefore, the reduced auto-arc space of $\bA_{k}^{1}$ at $p$ is defined by the equations $a_0=0$ and $(t\beta)^n=0$. However, the second equation is trivially satisfied -- i.e., $(t\beta)^n = t^n\cdot \beta^n= 0\cdot\beta^n = 0$ for any $\beta \in \kappa(p)[t]/(t^n)$. Equivalently, we have the following isomorphism
\begin{equation*}
\sA_n(\bA_{k}^1,p)^{\red}:=
(\nabla_{\fl_{\kappa(p),n}}\fl_{\kappa(p),n})^{\red}\cong \spec{\kappa(p)[a_0,\ldots,a_{n-1}]/(a_0)}.
\end{equation*}
for all $n\in\bN$. Thus, for all $n\in\bN$, we have 
\begin{equation*}
\sA_n(\bA_{k}^1,p)^{\red}\cong \bA_{\kappa(p)}^{n-1}.
\end{equation*}
\end{example}

\begin{lemma}\label{lem}
Let $p$ be a point of $\bA_{k}^{d}$. Then, for all $n\in\bN$,
\begin{equation*}
\sA_n(\bA_{k}^{d},p)^{\red} \cong \bA_{\kappa(p)}^{r_n}
\end{equation*}
where $r_n=d(\ell(J_p^n\bA_{k}^{d})-1)$.
\end{lemma}

\begin{proof}
By definition, the coordinate ring of $J_p^n\bA_{k}^{d}$ is isomorphic to
\newline $\kappa(p)[x_1,\ldots,x_d]/(x_1,\ldots,x_d)^n$. Thus, for each $i=1,\ldots,d$, we define $\alpha_i=\sum_{|j|<n}a_j^{(i)}x^j$ where $j$ is a multi-index (i.e., $j=(j_1,\ldots,j_d)$, $x^j = \prod_{s=1}^{d}x_s^{j_s}$, and $|j|=\sum_{s=1}^{d}j_s< n$) and where $a_j\in\kappa(p)$.
Then, the equations defining the auto-arc space are given by 
\begin{equation*}
0=\alpha_{i}^{n} = (a_{0}^{(i)} + \beta_i)^n, \quad \forall i =1,\ldots,d
\end{equation*}
where $0$ is treated as the multi-index $(0,\ldots,0)$ and 
$\beta_i\in(x_1,\ldots,x_d)/(x_1,\ldots,x_d)^n$. This implies 
that $0=(a_{0}^{(i)})^n$ for all $i=1,\ldots,d$ on 
$\sA_n(\bA_{k}^{d},p)$. Thus, in the reduction, 
$0=a_{0}^{(i)}$ for all $i=1,\ldots,d$. Clearly, $\beta^n=0$ 
for all $\beta\in (x_1,\ldots,x_d)/(x_1,\ldots,x_d)^n$. Thus, 
$\sA_n(\bA_{k}^{d},p)^{\red}$ is defined by the equations 
$0=a_{0}^{(i)}$ for all $i=1,\ldots,d$ while $a_{j}^{(i)}$ 
are free variables for $0< |j|<n$. 
Thus, $\sA_n(\bA_{k}^{d},p)^{\red}$ is isomorphic to $\bA_{\kappa(p)}^{r_n}$ for some non-negative integer $r_n$.  It is immediate then that $r_n$ is equal to  $d(\ell(J_p^n\bA_{k}^{d})-1)$.
\end{proof}

\begin{theorem}\label{them}
Let $X$ be an object of $\sch{k}$ and let $p$ be any point of 
$X$ such that $X$ is smooth at $p$. Then, for all  $n\in\bN$, there is a canonical isomorphism
\begin{equation*}\sA_{n}(X,p) \cong \sA_{n}(\bA_{\kappa(p)}^d,q)
\end{equation*}
where $d = \dim{}_{p}X = \mbox{krull-dim}(\sO_{X,p})$ and where $q$ is some $\kappa(p)$-rational point of $\bA_{\kappa(p)}^d$.
\end{theorem}

\begin{proof}
By the assumption that $X$ is smooth at $p$, there exists an open affine $U$ of $X$ containing the point $p$ and an \'{e}tale morphism $f: U \to \bA_{k}^{d}$ where $d = \dim{}_{p}X$, cf. Corollary 2.11 of \cite{Liu}. We let $p'\in\bA_{k}^{d}$ be such that $f(p)=p'$. Then, since $f$ is \'{e}tale, $\kappa(p)$  is a separable field extension of $\kappa(p')$ and the canonical ring homomorphism
\begin{equation*}
e:\hat\sO_{\bA_{k}^{d},p'} \otimes_{\kappa(p')}\kappa(p)\to \hat\sO_{X,p}
\end{equation*}
is an isomorphism, cf., Problem 10.4 of Chapter III, \S 10 of \cite{Ha1}. Thus, for each $n\in \bN$, we have that $e$ induces a morphism
\begin{equation*}
  e_n : J_{p}^{n}X \to J_{q}^{n}\bA_{\kappa(p)}^d, 
\end{equation*}
and, moreover, this is an isomorphism schemes. Note here that $q$ is the extension of $p'$ induced by the functor $-\otimes_{\kappa(p')}\kappa(p)$, and so, $\kappa(q) \cong \kappa(p)$, which proves the theorem.
\end{proof}

\begin{conjecture}
Let $X$ be an object of $\Var{k}$ and let $p$ be a point of $X$. If, for $n$ sufficiently large, $\sA_n(X,p)^{\red}$ is isomorphic to $\bA_{\kappa(p)}^{r_n}$ for some $r_n\in\bN$  then $X$ is smooth at $p$.
\end{conjecture}

\begin{remark}
Thus, we expect the condition that $\sA_n(X,p)^{\red}$ is isomorphic to affine $r_n$-space over $\kappa(p)$ for all large $n$ to be equivalent to smoothness of $X$ at $p$. This is part of a more general picture which the reader may find in \S\ref{zeta}, specifically in Conjecture \ref{etconj}.
\end{remark}

\section{The auto Igusa-zeta function}\label{zeta}

\begin{definition}\label{zeta0}
Let $X$ be an object of $\sch{k}$ and let $p$ be any point of $X$ with residue field $\kappa(p)$. Then, we define the {\it auto Igusa-Zeta series} associated to $X$ at $p$ to be the series
\begin{equation}\label{zeta1}
\zeta_{X,p}^{\auto}(t):= \sum_{n=0}^{\infty}[\sA_{n+1}(X,p)]\bL^{-\dim{}_p(X)\cdot\ell(J_{p}^{n+1}X)}t^n \in \sH_{\kappa(p)}[[t]].
\end{equation}
where $\dim{}_p(X) := \mbox{krull-dim}(\sO_{X,p}).$
Furthermore, we define the {\it reduced auto-Igusa-Zeta series} associated to $X$ at $p$ to be the series
\begin{equation}\label{zeta2}
\bar\zeta_{X,p}(t) := \sigma_{\kappa(p)}^{\prime}(\zeta_{X,p}^{\auto}(t)) \in \sG_{\kappa(p)}[[t]].
\end{equation}
\end{definition}

\begin{remark}
In this paper, we focus primarily on the function $\bar\zeta_{X,p}(t)$. The motivating case occurs when $X$ is an irreducible algebraic curve over $\bC$ and $p$ is the only singular point of $X$. In this case, $\bar\zeta_{X,p}(t)$  will be an element of $\sG_{\bC}[[t]]$.
\end{remark}

\begin{example}
The simplest case occurs when $X$ is $\spec{F}$ with $F$ any\footnote{Note that we defined the auto Igusa-Zeta function for schemes of finite type over $k$. However, at the very least, the definition clearly extends without issue to the case where $X$ is any zero dimensional scheme over $k$.} field extension of $k$. Then, $\sA_n(\spec{F},(0)) = \spec{F}$ for all $n\in\bN$. Thus, 
\begin{equation*}
\zeta_{\spec{F},(0)}^{\auto}(t) = \sum_{n=0}^{\infty}t^n = \frac{1}{1-t}
\end{equation*}
because $[\spec{F}] = 1$ in $\sH_F$ and the dimension of $\spec{F}$ is zero. Clearly then, $\bar\zeta_{\spec{F},(0)}(t)$ is also equal to $\frac{1}{1-t}$.
\end{example}

\begin{example}\label{firstex}
Using the notation of Definition \ref{zeta0}, let $X$ be $\bA_{k}^{1}$ and let $p$ be a point of $\bA_{k}^{1}$, then by Example \ref{linauto}, we have
\begin{equation*}
\begin{split}
\bar\zeta_{\bA_{k}^{1}, p}(t) &= \sum_{n=0}^{\infty}\sigma_{\kappa(p)}^{\prime}([\sA_{n+1}(\bA_{k}^1,p)]\bL^{-n-1})t^n \\
&= \sum_{n=0}^{\infty}[\sA_{n+1}(\bA_{k}^1,p)^{\red}]\cdot\sigma_{\kappa(p)}^{\prime}(\bL^{-n-1})t^n\\
&= \sum_{n=0}^{\infty}\bL^{n}\cdot\bL^{-n-1}t^n = \bL^{-1}\cdot\sum_{n=0}^{\infty}t^n.
\end{split}
\end{equation*}
Thus, for every point $p$ of $\bA_{k}^{1}$, we have \begin{equation}
\bar\zeta_{\bA_{k}^{1}, p}(t) = \bL^{-1}\cdot\frac{1}{1-t}.
\end{equation}
\end{example}

\begin{remark}
Note that the auto-arc spaces $\sA_{n}(\bA_{k}^{1}, p)$ are in fact relatively complicated\footnote{In fact, this complexity seems to grow quite rapidly as $n$ increases.} because of their non-reduced structure. Thus, I expect that in order for $\zeta_{X,p}^{\auto}(t)$ to be rational, $X$ must be a zero dimensional scheme over $k$; however, I do not have a proof of this fact. Nevertheless, the chief reason we introduced $\zeta_{X,p}^{\auto}(t)$ was so that we could introduce $\bar\zeta_{X,p}(t)$.
\end{remark}

\begin{example} By Lemma \ref{lem}, we may perform an entirely similar calculation as in Example \ref{firstex} to obtain
\begin{equation*}
\bar\zeta_{\bA_{k}^{d},p}(t) = \bL^{-d}\cdot \frac{1}{1-t}.
\end{equation*}
\end{example}

\begin{proposition}\label{smzeta}
Let $X$ be an object of $\sch{k}$ and let $p$ be a point of $X$. If $X$ is smooth at $p$, then
\begin{equation}
\bar\zeta_{X,p}(t) =\bL^{-d}\cdot\frac{1}{1-t}
\end{equation}
 where $d=\dim{}_pX$. 
\end{proposition}

\begin{proof}
This follows immediately from Theorem \ref{them}.
\end{proof}

The immediate generalizations of Theorem \ref{them} and the previous proposition do in fact hold.  Thus, it is clear that if $f:X\to Y$ is an \'{e}tale morphism at a point $p\in X$, then  
\begin{equation}\label{et}
\bar\zeta_{X,p}(t) = \bar\zeta_{Y,f(p)}(t) 
\end{equation}
 Therefore, it is natural to conjecture that the converse is also true--i.e., we posit the following conjecture.

\begin{conjecture}\label{etconj}
 If the power series $$\bar\zeta_{X/Y,p}(t):=\bar\zeta_{X,p}(t)  -  \bar\zeta_{Y,q}(t)$$ is  equal to zero for some $q\in Y$, then $(X,p)$ is analytically isomorphic to $(Y,q)$.
\end{conjecture}

\begin{remark}
In fact, we expect that $(X,p)$ is analytically isomorphic to $(Y,q)$ if and only if  $\sA_n(X,p)^{\red}$ is isomorphic to $\sA_n(Y,q)^{\red}$ for sufficiently large $n\in\bN$. As we have already noted, a simple extension of the argument used in the proof of Theorem \ref{them} proves the forward direction of this statement.
\end{remark}

\section{Computing auto-arc spaces with Sage}\label{compsec}

In this section, we will run some code that the author programmed in Sage which computes the Arc space $\nabla_{\fn}X$ of an affine scheme $X\in \sch{k}$ where $\fn\in\fatpoints{k}$ and $k$ is a field of characteristic zero. The code can be found in Appedix C.

\subsection{Computation for cuspidal cubic:}

Let us consider the cuspidal cubic $C =\spec{k[x,y]/(y^2+x^3)}$. We can run the sage script to compute the auto-arcs $\sA_n(C,O)$ for $n\leq 6$ where $O$ is the origin. Note that the complexity of these spaces grows quite rapidly because $\sA_n(C,O)$ is not reduced. This will be made evident in the following. 

Lets call the coordinate ring of this affine scheme $A_n$. For $n=1$, we obtain $A_n=k$ as always. For $n=2,$ the sage output gives the list of equations:

\medskip

\begin{center}
\begin{tabular}{ |l |  l|}
\hline
\multicolumn{2}{|c|}{\mbox{Equations for $\sA_2(C,O)$}}  \\
  \hline                       
  1. $a_0=a_1=a_2=a_3=0$ &4. $a_7a_8 + a_6a_9=0$  \\
  \hline
  2. $2a_5a_9=2a_7a_9=0$ &5. $2a_4a_8=2a_6a_8=0$   \\
  \hline
  3. $a_5a_8 + a_4a_9=0$ &6. $a_8a_9=a_8^2=a_9^2=0$   \\
  \hline 
\end{tabular}
\end{center}
\medskip

\noindent which take place in $k[a_0,a_1,\ldots,a_9]$. Note that I have manually rendered the sage output, which can either be a python list, a sage ideal, or a singular quotient ring. The $6$th equation will reduce to $a_9 = 0$ and $a_8=0$, in $A_2/\mbox{nil}(A_2)$, and as either $a_8$ or $a_9$ occurs in each term of each equation, we have that the variables $a_i$ are free for $i = 4,$ $5,$ $6,$ $7$. In other words, $A_2/\mbox{nil}(A_2)$ is isomorphic to $k[a_4,a_5,a_6,a_7]$, and the reduced auto-arc space $\sA_n(C,O)^{\red}$  is isomorphic to $\bA_{k}^{4}$. 

For $n=3$, our manually rendered sage output is

\begin{center}
\begin{tabular}{ |l |  l|}
\hline
\multicolumn{2}{|c|}{\mbox{Equations for $\sA_3(C,O)$}}  \\
  \hline                       
  1. $a_0=a_1=a_2=a_3=0$ &7. $a_7a_{12}^{2} + 2a_6a_{12}a_{13}=0$ \\
  \hline
  2. $2a_7a_{11}+2a_5a_{13}=0$ &8. $a_{11}a_{12}^{2} + 2a_{10}a_{12}a_{13}=0$    \\
  \hline
  3. $2a_7a_{13}=2a_{11}a_{13}=0$ & 9. $ 6a_6a_{10}a_{12} + 3a_4a_{12}^{2}+2a_7a_{11} + 2a_5a_{13}=0$ \\
  \hline 
  4. $a_{11}^{2} + 2a_{9}a_{13}=0$  & 10. $3a_{10}^{2}a_{12} + 3a_8a_{12}^{2} + a_{11}^{2} + 2a_9a_{13}=0$ \\
  \hline
  5. $2a_7a_{10}a_{12} + 2a_6a_{11}a_{12}+ a_5a_{12}^{2}+\cdots$ &11. $3a_{6}a_{12}^{2} + 2a_7a_{13}=0$\\ 
  \cline{2-2}
   $\quad \quad \quad \quad \quad \quad \cdots+ 2a_6a_{10}a_{13} +2a_4a_{12}a_{13}=0$ &12. $3a_{10}a_{12}^{2} + 2a_{11}a_{13}=0$ \\
  \hline 
  6. $2a_{10}a_{11}a_{12}+a_{9}a_{12}^{2}+a_{10}^{2}a_{13}+2a_{8}a_{12}a_{13}=0$ & 13. $a_{12}^{3} + a_{13}^{2}=a_{12}^{2}a_{13}=a_{13}^{2}=0$, \\
  \hline
\end{tabular}
\end{center}

\medskip

\noindent The equations above take place in $k[a_0,a_1,\ldots, a_{13}]$. One can already see that the list of equations grows rapidly. Here, the first equation which tells us that the first $4$ variables are evaluated at zero has to do with the way I wrote the program and tells us nothing substantive mathematically. What we may notice is that equations $10$ and $13$  tell us that $a_{11} = 0,$ $a_{12}=0, $ and $a_{13}=0$ in $A_3/\mbox{nil}(A_3)$. Then, one may manually check that one of these variables occurs at least once in each term of each equation, as before. Thus, the variable $a_i$ are free for $i = 4, 5, \ldots, 10$, or, in other words, the reduced auto-arc scheme $\sA_3(C,O)^{\red}$ is isomorphic to $\bA_{k}^{7}$. 

For $n=4$, we will see that $\sA_{4}(C,O))^{\red}$ is not smooth. In fact, we will verify Schoutens' claim (cf. Example 4.17 of \cite{Sch2}) that it is isomorphic to $\nabla_{\fl_2}C \times_{k}\bA_{k}^7$. Indeed, the sage script will verify this (or, this will ease the sceptical reader into believing that the code works well). For $n=4$, we have 

\begin{center}
\begin{tabular}{ |l |  l|}
\hline
\multicolumn{2}{|c|}{\mbox{Equations for $\sA_4(C,O)$}}  \\
\hline
1.  $a_0=a_1=a_2=a_3=0$ &9. $2a_{11}a_{16}a_{17} + a_{10}a_{17}^{2}=0$\\
\hline
2. $-a_{15}^{3} + 3a_{11}^{2}a_{17} - 6a_{13}a_{15}a_{17} + 3a_5a_{17}^{2}=0$ &10. $a_{15}^2a_{16} + 2a_{14}a_{15}a_{17} + 2a_{13}a_{16}a_{17} +a_{12}a_{17}^{2}=0$ \\
\hline
3. $ 3a_{11}a_{15}^{2} + 6a_{11}a_{13}a_{17} + 6a_9a_{15}a_{17} + 3a_{7}a_{17}^{2}=0$  &11. $2a_{15}a_{16}a_{17} + a_{14}a_{17}^2=0$ \\ 
\hline
4.  $6a_{11}a_{15}a_{17} + 3a_{9}a_{17}^{2}=3a_{11}a_{17}^{2}=0$ & 12. $-a_{14}^{3} + 3a_{10}^{2}a_{16} - 6a_{12}a_{14}a_{16}+\cdots$\\
\cline{1-1}
5. $3a_{15}^{2}a_{17} + 3a_{13}a_{17}^{2}=3a_{15}a_{17}^{2}=0$ & $\quad\quad\quad \cdots+ 3a_{4}a_{16}^{2}  + a_{11}^2 - 2a_{13}a_{15} + 2a_{5}a_{17}=0$\\
\hline
6. $-a_{14}a_{15}^{2} + a_{11}^{2}a_{16} - 2a_{13}a_{15}a_{16}+\cdots$ &13. $3a_{10}a_{14}^2 + 6a_{10}a_{12}a_{16}+ 6a_{8}a_{14}a_{16} +\cdots$\\ 
$\quad\quad\cdots+2a_{10}a_{11}a_{17}  - 2a_{13}a_{14}a_{17}-\cdots$ &$\quad\quad\quad+3a_6a_{16}^2 + 2a_{11}a_{13} + 2a_9a_{15} + 2a_7a_{17}=0$\\ \cline{2-2}
$\quad\quad\quad\cdots-2a_{12}a_{15}a_{17} + 2a_{5}a_{16}a_{17}=0 + a_4a_{17}^2=0$ &14. $6a_{10}a_{14}a_{16} + 3a_8a_{16}^2 + 2a_{11}a_{15} + 2a_9a_{17}=0$\\ 
\hline
7.    $2a_{11}a_{14}a_{15} + a_{10}a_{15}^{2} + 2a_{11}a_{13}a_{16} +\cdots$ &15. $3a_{10}a_{16}^2 + 2a_{11}a_{17}=0$\\
\cline{2-2}
$\quad\quad\cdots+2a_9a_{15}a_{16} + 2a_{11}a_{12}a_{17} +\cdots$ &16. $3a_{14}^2a_{16} + 3a_{12}a_{16}^2 + a_{15}^2 + 2a_{13}a_{17}=0$\\
\cline{2-2}
$ \quad\quad\cdots+ 2a_{10}a_{13}a_{17} + 2a9a_{14}a_{17} + 2a_8a_{15}a_{17} +\cdots$ &17. $3a_{14}a_{16}^2 + 2a_{15}a_{17}=0$\\
\cline{2-2}
$\quad\quad\quad\quad\cdots +2a_7a_{16}a_{17} + a_{6}a_{17}^{2}=0$ &18. $a_{12}^{2}a_{13}=0$\\
\hline
 8.  $2a_{11}a_{15}a_{16} + 2a_{11}a_{14}a_{17} +\cdots$&19. $a_{16}^{3} + a_{17}^{2}=0$\\ 
 \cline{2-2}
 $\quad\quad\quad\quad+\cdots 2a_{10}a_{15}a_{17}+2a_9a_{16}a_{17} + a_8a_{17}^{2}=0$ &20. $a_{17}^3=0$\\
 \hline
 \end{tabular}
 \end{center}
 
\medskip 
 
\noindent These equations take place in $k[a_0, a_1,\ldots, a_{17}]$. It is completely obvious now that putting the full list of equations for this auto-arc space when $n=5$ and $n=6$ is untenable. However, given that I am providing the code to my program, it is unnecessary to do so. We are doing it for $n=4$ so that the reader may see how things work in practice. At any rate (later we will see this is a general pattern), we notice that equations 16, 19, and 20 show us that $a_{15}=a_{16}=a_{17}=0$ in $A_4/\mbox{nil}(A_4)$. Thus, the list of equations defining $A_4/\mbox{nil}(A_4)$ will be
\begin{equation}\label{tag}
\begin{split}
\mbox{1.  } & a_0=a_1=a_2=a_3=a_{15}=a_{16}=a_{17}=0,\\
  \mbox{2. } & -a_{14}^{3} + a_{11}^2 =0, \\
   \mbox{3. } & 3a_{10}a_{14}^2  + 2a_{11}a_{13} = 0,  
\end{split}
\end{equation}
which takes place in $k[a_0, a_1,\ldots, a_{17}]$. We reached these equations by noticing that the only equations with terms not involving $a_{15},$ $a_{16},$ or $a_{17}$ are equations 12 and 13, and those two equations simplify to equation 2 and 3 above, respectively. So, here, the variables $a_4, \ a_5, \ a_6, \ a_7, \ a_8, \ a_9$ and $a_{12}$ are free so that $A_4/\mbox{nil}(A_4)$ is the tensor product of a multivariate polynomial ring in $7$ variables over $k$ with $S$ where $S$ is the quotient  ring of a multivariate polynomial ring in $4$ variables by equations $2$ and $3$ of (\ref{tag}). One may quickly check that $S$ is isomorphic to the coordinate ring of the arc space $\nabla_{\fl_2}C$. Thus, Schoutens' statement is verified -- i.e., for $n=4$, we have an isomorphism 
\begin{equation*}
\sA_{4}(C,O)^{\red} \cong \nabla_{\fl_2}C \times_{k} \bA_{k}^{7} .
\end{equation*} 
For $n=5$, we will work with the coordinate ring of the reduction $B_5:=A_5/\mbox{nil}(A_5)$. I choose to do this by hand as, at least as far as I understand, reduction is not fully implemented in Sage or Singular. Similar to before, we will work in the multivariate polynomial ring $k[a_0,a_1,\ldots, a_{21}]$, and we get a long list of equations in which it is easy to show that $a_{19}=a_{20}=a_{21}=0$ in $B_5$. Then, it is easy to reduce and find that the equations defining $B_5$. These equations are given by the following list.
\begin{equation*}
\begin{split}
 \mbox{1.  } & a_0=a_1=a_2=a_3=a_{19}=a_{20}=a_{21}=0, \\
 \mbox{2.  } & a_{14}^3 - 6a_{14}a_{16}a_{18} - 3a_{12}a_{18}^2 + 2a_9a_{15} - 2a_{13}a_{17} = 0 , \\
 \mbox{3.  } & 3a_{14}^2a_{18} - 3a_{16}a_{18}^2 + 2a_{13}a_{15} - a_{17}^2= 0, \\
 \mbox{4.  } & 3a_{14}a_{18}^2+ 2a_{15}a_{17}= 0 , \\
 \mbox{5.  } & -a_{18}^3 +a_{15}^2= 0, 
\end{split}
\end{equation*}
which take place in $k[a_0,a_1,\ldots, a_{21}]$. One quickly notices that $a_4, \ a_5, \ a_6, \ a_7, \ a_8, \ a_{10},$ and $a_{11}$ do not occur in the aforementioned equations. One may check rather quickly (either by running the program or by hand) that equations 2-5 define the coordinate ring of the arc space $\nabla_{\fl_4}C$. Thus, we have shown that 
\begin{equation*}
\sA_5(C,O)^{\red} \cong \nabla_{\fl_4} C \times_{k} \bA_{k}^{7}.
\end{equation*}

Although the complexity increases drastically, the case for $n=6$ is exactly the same. I personally verified using my sage script that 
\begin{equation*}
\sA_6(C,O)^{\red} \cong \nabla_{\fl_6} C \times_{k} \bA_{k}^{7}, 
\end{equation*}
when $n=6$. It is more complicated, but I am confident that the interested reader could do the same calculation using the code.  At any rate, this leads us to conjecture that for $n \geq 4$, we have the following isomorphism 
\begin{equation*}
\sA_n(C,O)^{\red} \cong \nabla_{\fl_{2(n-3)}} C \times_{k} \bA_{k}^{7}.
\end{equation*}
We will offer a proof of this fact in the next section. The reason I carried out the calculation here when I already arrived at the proof is for two reason. First, this is how I arrived at the result, and secondly, it demonstrates how this can be done for other auto-arc spaces. It appears that after computing the first few auto-arc spaces with my program (in general, this will take calculations which cannot be done in any reasonable sense by hand), one will be able to see a general pattern and be able to make a conjecture. Then, at least in my experience so far, a proof can be obtained. 

\subsection{Computation for node:}
Perhaps a more manageable computation occurs when 
$N = \spec{k[x,y]/(xy)}$. So, let $A_n$ be the coordinate ring of $\sA_n(N,O)$ where $O$ is the origin. Let $B_n$ be the coordinate ring of the reduction. As usual, we have $A_1=B_1 = \spec{k}$. For $n=2$, the sage script gives the equations

\begin{center}
\begin{tabular}{ |l |  l|}
\hline
\multicolumn{2}{|c|}{\mbox{Equations for $\sA_2(N,O)$}} \\
\hline
1. $a_0=a_1=a_2=a_3 =0$ & 6.   $2a_4a_8 =0$\\
\hline
2. $2a_5a_9 =0$ & 7. $2a_6a_8 =0$\\
\hline
3. $2a_7a_9 =0$ & 8. $a_{9}^{2} =0$\\
\hline
4. $a_5a_8 + a_4a_9 =0$ &9. $a_8a_9 =0$\\
\hline 
5.  $a_7a_8 + a_6a_9 =0$ &10. $a_{8}^{2} =0$\\
\hline
\end{tabular}
\end{center}

\noindent These equations take place in $k[a_0,a_1, \ldots, a_9]$ and define $A_2$. So that the $a_8=a_9=0$ in $B_2$ so that $B_2 = k[a_4, a_5, a_6, a_7]$ and $\sA_2(N,O)^{\red} \cong \bA_{k}^{4}$. For $n=3$, we obtain the equations 

\begin{center}
\begin{tabular}{ |l |  l|}
\hline
\multicolumn{2}{|c|}{\mbox{Equations for $\sA_3(N,O)$}} \\
\hline
1. $a_0=a_1=a_2=a_3=0$  &8.  $3a_{11}a_{13}^{2}=0$\\
\hline
2. $a_6a_7 + a_5a_{12} + a_{4}a_{13}=0$  &9. $3a_{6}^{2}a_{12} + 3a_{4}a_{12}^{2}=0$\\
\hline
3. $a_7a_{12} + a_6a_{13}=0$ &10. $3a_{6}a_{12}^2=0$\\
\hline
4. $a_{10}a_{11} + a_9a_{12} + a_8a_{13}=0$ & 11. $3a_{10}^2a_{12} + 3a_8a_{12}^2=0$ \\ 
\hline
5. $a_{11}a_{12} + a_{10}a_{13}=0$  &12. $3a_{10}a_{12}^2=0$\\
\hline
6.  $3a_{7}^{2}a_{13} + 3a_5a_{13}^{2}=0$ &13. $a_{12}a_{13}=0$ \\
\hline
7. $3a_{11}^{2}a_{13} + 3a_9a_{13}^{2}=0$ & 14.  $a_{12}^3=a_{13}^3=0$\\
\hline
\end{tabular}
\end{center}

\noindent These equations take place in $k[a_0,a_1,\ldots,a_{13}]$. In $B_3$, we may reduce this list to 
\begin{equation*}
\begin{split}
\mbox{1.  } & a_0=a_1=a_2=a_3=a_{12}=a_{13}=0, \\
\mbox{2.  } & a_6a_7 = 0, \\
\mbox{3.  } & a_{10}a_{11}=0 . 
\end{split}
\end{equation*}
Thus, $$\sA_2(N,O)^{\red} \cong N\times_{k}N \times_{k} \bA_{k}^{4}.$$

For $n=4$, the  sage output is the following.

\begin{center}
\begin{tabular}{ |l |  l|}
\hline
\multicolumn{2}{|c|}{\mbox{Equations for $\sA_4(N,O)$}} \\
\hline
1. $a_0=a_1=a_2=a_3=0$ &12. $6a_{15}^2a_{17}^2 + 4a_{13}a_{17}^3=0$\\
\hline
2.  $a_7a_8 + a_6a_9 + a_5a_{16} + a_4a_{17}=0$  &13.  $4a_{15}a_{17}^3=0$\\
\hline
3.  $a_8a_9 + a_7a_{16} + a_6a_{17}=0$ &14.  $4a_{8}^3a_{16} + 12a_6a_8a_{16}^2 + 4a_4a_{16}^3=0 $\\
\hline 
4.   $a_9a_{16} + a_8a_{17}=0$ &15. $6a_{8}^2a_{16}^2 + 4a_{6}a_{16}^3=0$\\ 
\hline
5.  $a_{13}a_{14} + a_{12}a_{15} + a_{11}a_{16} + a_{10}a_{17}=0$ &16.  $4a_8a_{16}^3=0$\\ 
\hline
6.   $a_{14}a_{15} + a_{13}a_{16} + a_{12}a_{17}=0$ &17. $4a_{14}^3a_{16} + 12a_{12}a_{14}a_{16}^2 + 4a_{10}a_{16}^3=0$\\ 
\hline
7.  $a_{15}a_{16} + a_{14}a_{17}=0$ &18. $6a_{14}^2a_{16}^2 + 4a_{12}a_{16}^3=0$\\ 
\hline
8.  $4a_{9}^3a_{17} + 12a_{7}a9a_{17}^2 + 4a_5a_{17}^3=0$ &19. $4a_{14}a_{16}^3=0$\\ 
\hline
9.  $6a_{9}^2a_{17}^2 + 4a_7a_{17}^3=0$ &20.  $a_{16}a_{17}=0$\\
\hline
10. $4a_9a_{17}^3=0$ &21.  $ a_{16}^4=0$\\ 
\hline
11. $4a_{15}^3a_{17} + 12a_{13}a_{15}a_{17}^2 + 4a_{11}a_{17}^3=0$ &22.   $a_{17}^4=0 $\\ 
\hline
\end{tabular}
\end{center}

\medskip

\noindent These equations take place in $k[a_0, a_1, \ldots, a_{17}]$ and describes $A_4$. From this, one gathers that $a_{16}=a_{17} =0$ in $B_4$. Thus, the equations defining $B_4$ are 
\begin{equation*}
\begin{split}
\mbox{1.  } & a_0=a_1=a_2=a_3=a_{16}=a_{17}=0, \\
\mbox{2.  } & a_7a_8 + a_6a_9=0, \\
\mbox{3.  } & a_8a_9 =0,\\  
\mbox{4.  } & a_{13}a_{14} + a_{12}a_{15}=0,\\ 
\mbox{5.  } & a_{14}a_{15} =0,\\ 
\end{split}
\end{equation*}
From this one sees that $a_4, \ a_5, \ a_{10}, $ and $a_{11}$ are free in $B_4$ and that equations 2 and 3 have no variables in common with equations 4 and 5. In fact, equations 2 and 3 are the same as those which define the arc space $\nabla_{\fl_2}N$, and likewise, equations 4 and 5 are also those which define the arc space $\nabla_{\fl_2}N$. Thus, in the case where $n=4$, we arrive at 
\begin{equation*}
\sA_4(N,O)^{\red} \cong \nabla_{\fl_{2}}N\times_{k}\nabla_{\fl_{2}}N \times_{k} \bA_{k}^{4}.
\end{equation*}
In exactly the same way, I used the sage script to find an isomorphism   
\begin{equation*}
\sA_n(N,O)^{\red}\cong \nabla_{\fl_{n-2}}N\times_{k}\nabla_{\fl_{n-2}}N \times_{k} \bA_{k}^{4}.
\end{equation*}
for $n=5,\ldots, 8$. 
The only reason I do not include the calculation here is that it is too lengthy for the uninterested reader and can easily be checked by running the sage script for the interested reader. 
So, in the end, the above isomorphism is expected to hold for all $n$ greater than or equal to $3$, which is a fact we will prove in the next section.

\section{Proofs for the patterns notice in \S\ref{compsec}}

Per our calculations in \S\ref{compsec}, we posit the following theorem.
\begin{theorem}\label{cusp}
Let $k$ be any field such that $\mbox{char}(k)\neq 2, 3$.
Let $C \cong \spec{k[x,y]/(y^2-x^3)}$ and let $O\in C$ be the point at the origin. Then, for all $n\geq 4$, 
\begin{equation*}
\sA_n(C,O)^{\red} \cong (\nabla_{\fl_{2(n-3)}}C )\times_{k} \bA_{k}^{7}
\end{equation*} 
\end{theorem}
\begin{proof}
First, note that 
\begin{equation*}
(x,y)^n+(y^2-x^3) = (x^n, x^{n-1}y, y^2-x^3).
\end{equation*}
as ideals in $k[x,y]$.
Thus, we must define two arcs
\begin{equation}
\begin{split}
\alpha &:= \sum_{i=0}^{n-1} a_ix^i + \sum_{i=0}^{n-2} b_i x^iy \\
\beta &:= \sum_{i=0}^{n-1} c_ix^i + \sum_{i=0}^{n-2} d_i x^iy \\
\end{split}
\end{equation}
where $a_i, \ b_i, \ c_i$ and $d_i$ are thought of as variables running through $k$. We then have the following equations 
\begin{equation}
\alpha^n = \alpha^{n-1}\beta= \beta^2-\alpha^3 =0
\end{equation}
occurring in $$R:=k[a_i, c_i , b_j, d_j \mid i = 0, \ldots n-1, j=0,\ldots, n-2]\tensor_{k} k[x,y]/(x^n, x^{n-1}y, y^2-x^3)$$ where we think of $R$ as a finitely generated $k[x,y]/(x^n, x^{n-1}y, y^2-x^3)$-algebra. 
Now $0=\alpha^n = a_{0}^{n}$ and in the reduce structure this implies $a_0 = 0$. Likewise, $0= \alpha^{n-1}\beta$ implies $0=\beta^n = c_{0}^{n}$ since  $\alpha^3=\beta^2$. Thus, $c_0 = 0$ in the reduction. Thus, the equations $\alpha^n = \alpha^{n-1}\beta = 0$ are trivially satisfied in the reduced structure.
Now, we consider the following equation:
\begin{equation*}
\begin{split}
0=\beta^2 - \alpha^3 &= (\sum_{i=1}^{n-1} a_ix^i + \sum_{i=0}^{n-2} b_i x^iy)^2 - (\sum_{i=1}^{n-1} c_ix^i + \sum_{i=0}^{n-2} d_i x^iy)^3 \\
& =  (\sum_{i=1}^{n-1} a_ix^i)^2 +2(\sum_{i=1}^{n-1} a_ix^i)\cdot(\sum_{i=0}^{n-2} b_i x^iy) + (\sum_{i=0}^{n-2} b_i x^iy)^2 - \\
& - (\sum_{i=1}^{n-1} c_ix^i)^3 - 3(\sum_{i=1}^{n-1} c_ix^i)^2\cdot(\sum_{i=0}^{n-2} d_i x^iy) - \\
& -  3(\sum_{i=1}^{n-1} c_ix^i)\cdot(\sum_{i=0}^{n-2} d_i x^iy)^2 - (\sum_{i=0}^{n-2} d_i x^iy)^3 \\
\end{split}
\end{equation*}
Note the following identities involving each term of the above. 
\begin{equation*}
\begin{split}
 (\sum_{i=1}^{n-1} a_ix^i)^2& = x^2(\sum_{i=0}^{n-2} a_{i+1}x^i)^2 \\
 2(\sum_{i=1}^{n-1} a_ix^i)\cdot(\sum_{i=0}^{n-2} b_i x^iy) & = 2yx(\sum_{i=0}^{n-2} a_{i+1}x^i)\cdot(\sum_{i=0}^{n-2} b_i x^i) \\
(\sum_{i=0}^{n-2} b_i x^iy)^2 & = x^3(\sum_{i=0}^{n-2} b_i x^i)^2 \\
(\sum_{i=1}^{n-1} c_ix^i)^3 & = x^3(\sum_{i=0}^{n-2} c_{i+1}x^i)^3\\
 3(\sum_{i=1}^{n-1} c_ix^i)^2\cdot(\sum_{i=0}^{n-2} d_i x^iy)& = 3x^2y(\sum_{i=0}^{n-2} c_{i+1}x^i)^2\cdot(\sum_{i=0}^{n-2} d_i x^i) \\
3(\sum_{i=1}^{n-1} c_ix^i)\cdot(\sum_{i=0}^{n-2} d_i x^iy)^2 & = 3x^4(\sum_{i=0}^{n-2} c_{i+1}x^i)\cdot(\sum_{i=0}^{n-2} d_i x^i)^2
\end{split}
\end{equation*}
Thus, we group the terms involving $y$ and the terms only involving $x$ to obtain the equation
\begin{equation*}
\begin{split}
0 = &x^2(\sum_{i=0}^{n-2} a_{i+1}x^i)^2 + x^3((\sum_{i=0}^{n-2} b_i x^i)^2-(\sum_{i=0}^{n-2} c_{i+1}x^i)^3)-3x^4(\sum_{i=0}^{n-2} c_{i+1}x^i)\cdot(\sum_{i=0}^{n-2} d_i x^i)^2 + \\
& \ \ + 2xy(\sum_{i=0}^{n-2} a_{i+1}x^i)\cdot(\sum_{i=0}^{n-2} b_i x^i) - 3x^2y(\sum_{i=0}^{n-2} c_{i+1}x^i)^2\cdot(\sum_{i=0}^{n-2} d_i x^i)
\end{split}
\end{equation*}
From this we can see that the coefficient of $x^{2}$ is $a_{1}^2$. This implies $a_{1}^{2} = 0$ and so in the reduction $a_1 = 0$. Thus, we rewrite the previous equation as 
\begin{equation*}
\begin{split}
0 = &x^4(\sum_{i=0}^{n-3} a_{i+2}x^i)^2 + x^3((\sum_{i=0}^{n-2} b_i x^i)^2-(\sum_{i=0}^{n-2} c_{i+1}x^i)^3)-3x^4(\sum_{i=0}^{n-2} c_{i+1}x^i)\cdot(\sum_{i=0}^{n-2} d_i x^i)^2 + \\
& \ \ + 2x^2y(\sum_{i=0}^{n-3} a_{i+2}x^i)\cdot(\sum_{i=0}^{n-2} b_i x^i) - 3x^2y(\sum_{i=0}^{n-2} c_{i+1}x^i)^2\cdot(\sum_{i=0}^{n-2} d_i x^i)
\end{split}
\end{equation*}
 Furthermore, from this, we should be able to find exactly $7$ free variables. In fact, it is easy to see that $a_{n-1}, \ c_{n-1}, \ c_{n-2}, \ b_{n-2}, \ b_{n-3}, \ d_{n-2}, \ $ and $d_{n-3}$ are the only free variables. Thus, the space is of the form $S\times_{k}\bA_{k}^{7}$ where $S$ is an algebraic variety in the variables $a_i, \ b_j , \ c_l, \ d_k$ for appropriate indices $i, \ j, \ l, \ k$. Thus, going back to the original equations and setting these free variables to zero, we have 
\begin{equation*} 
 0= (\sum_{i=0}^{n-4} a_{i+2}x^i + \sum_{i=0}^{n-4} b_i x^iy)^2 - (\sum_{i=0}^{n-4} c_{i+1}x^i + \sum_{i=0}^{n-4} d_i x^iy)^3
 \end{equation*}
We may perform the  substitution $x=s^2$ and $y=s^{3}$ to obtain an equivalent equation, which defines $S$ as a reduced subscheme of $\nabla_{\fl_{2n-4}}C$.
This equation is of the following form:
 \begin{equation*}
 \begin{split}
 0 = & (a_2+s^{2}\sum_{i=0}^{2n-7}\sigma_is^i + b_{n-4}s^{2n-5})^2 - (c_1+\sum_{i=0}^{2n-7}\nu_is^i +d_{n-4}s^{2n-5})^3 \\
 = & a_2^2 + 2a_2s^2\sum_{i=0}^{2n-7}\sigma_is^i +2a_2b_{n-4}s^{2n-5}+s^{4}(\sum_{i=0}^{2n-7}\sigma_is^i)^{2} - \\
 & - c_1^{3} - 3c_1^2s^2\sum_{i=0}^{2n-7}\nu_is^i - 3c_1^2d_{n-4}s^{2n-5} - 3c_1s^4(\sum_{i=0}^{2n-7}\nu_is^i)^2- s^{6}(\sum_{i=0}^{2n-7}\nu_is^i)^3 \ .
 \end{split}
 \end{equation*}
 The above equation is obtained from the previous one by substituting $x=s^2$ and $y=t^{3}$, defining 
\begin{equation*}
\begin{split}
\sigma_i =
\left\{
	\begin{array}{ll}
		a_{i/2+3}  & \mbox{if } i=0,2,\ldots,2n-6 \\
		b_{i/2-1/2} & \mbox{if } i=1,3,\ldots,2n-7
	\end{array}
\right.\\
\nu_i =
\left\{
	\begin{array}{ll}
		c_{i/2+2}  & \mbox{if } i=0,2,\ldots,2n-6 \\
		d_{i/2-1/2} & \mbox{if } i=1,3,\ldots,2n-7 \
	\end{array}
\right.
\end{split}
\end{equation*} 
and, of course, we also expanded the product.

Now, let $\rho_m : \nabla_{\fl_{m}}C \to C$ be the natural truncation morphism induced by $k[t]/(t^m)\to k$.
The key point to notice is that $S$ contains the subscheme $\rho_{2n-6}^{-1}(O)$ where $O$ is the singularity of the cusp $C$. In fact, in $\rho_{2n-3}^{-1}(O)\cap S$, is defined by the equation above by setting $a_2=c_1=0$. Thus, it is defined by 
\begin{equation*}
\begin{split}
0=&s^{4}(\sum_{i=0}^{2n-7}\sigma_is^i)^{2}- s^{6}(\sum_{i=0}^{2n-7}\nu_is^i)^3 \\
=& s^{2}(s\sum_{i=0}^{2n-7}\sigma_is^i)^{2}- s^{3}(s\sum_{i=0}^{2n-7}\nu_is^i)^3
\end{split}
\end{equation*}
which is exactly the equation for $\rho_{2n-6}^{-1}(O)$. Thus, the restriction of the natural truncation morphism gives a morphism $f$ from $S$ to  $\nabla_{\fl_{2n-6}}C$ which is surjective and in fact an isomorphism on the inverse images of the singular point of $C$. However, away from these inverse images of the singular point, the morphism $f$ is a piecewise trivial fibration, cf. (2.7) \S 2 of \cite{DL1}. Moreover, the most the dimension of the trivial fiber can be is $2$ as $\dim(C) = 1$ and $(2n-4)-(2n-6) = 2$. However, $S$ as a subvariety of $\nabla_{\fl_{2n-4}}C$ is cut out by two hyperplanes -- i.e., the coefficient in front of the $t^{2n-6}$ term of each arc is zero. This means that this fiber must actually have dimension $0$ -- i.e., it is an isomorphism away from the singular locus. Thus, on the singular locus $f$ is an isomorphism and away from the singular locus $f$ is an isomorphism. Therefore, $f$ is an isomorphism, which proves the claim. 
 \end{proof}

\begin{remark}
I cannot see any difficultly with extending the above argument to other types of cusps. Thus, 
it should be expected to find a similar formula for the reduced auto-arc spaces of the curve $C(m,l)=\spec{k[x,y]/(y^l - x^m)}$ where $m> l$ (provided that $\mbox{char}(k)\nmid m, l$). It should be expected that there is an isomorphism $$\sA_n(C(m,l),O)^{ \red}\cong(\nabla_{\fl_{l(n-m)}}C(m,l))\times_{k} \bA_{k}^{r}$$ for some fixed $r \in \bN$ whenever $n> m$. Perhaps, one may also be able to show that $r$ is equal to $ml+1$. In particular, it is expected that the asymptotic defect of $J_{O}^{\infty}C(m,l)$ is given by
\begin{equation*}
\delta(J_{O}^{\infty}C(m,l)) := \limsup_{n} \frac{\dim{\sA_n(C(m,l),O)}}{\ell(J_O^nC(m,l))} = 2 
\end{equation*}
Thus, we often expect the asymptotic defect of a germ of an irreducible curve at a point to be equal to its embedding dimension. More general conjectures regarding asymptotic defects can be found in Chapter $5$ of \cite{Sch2}.
\end{remark}

\begin{theorem}\label{node} Let $k$ be any field.
Let $N$ be isomorphic to $\spec{k[x,y]/(xy)}$, and let $O\in N$ be the point at the origin. Then, for each $n\geq 3$, we have an isomorphism
\begin{equation*}
\sA_n(N,O)^{ \red} \cong \nabla_{\fl_{n-2}}N\times_{k}\nabla_{\fl_{n-2}}N\times_{k} \bA_{k}^{4} .
\end{equation*}
\end{theorem}

\begin{proof}
As in the case of the previous proof, we again define two arcs 
\begin{equation*}
\begin{split}
\alpha & := \sum_{i=0}^{n-1} a_ix^i + \sum_{i=1}^{n-1} b_iy^i \\
\beta  & := \sum_{i=0}^{n-1} c_ix^i + \sum_{i=1}^{n-1} d_iy^i
\end{split}
\end{equation*}
and investigate the equations
\begin{equation*}
0=\alpha^n=\beta^n = \alpha\beta \ .
\end{equation*}
Note that
\begin{equation*}
\begin{split}
0 = \alpha^n = a_{0}^{n} \implies 0 = a_{0} \mbox{ in the reduction and} \\
0=\beta^n = c_{0}^n \implies 0 = c_0 \mbox{ in the reduction.}
\end{split}
\end{equation*}
Thus, we only have to investigate 
\begin{equation*}\begin{split}
0=\alpha\beta &= (\sum_{i=1}^{n-1} a_ix^i)(\sum_{i=1}^{n-1} c_ix^i)+(\sum_{i=1}^{n-1} a_ix^i)(\sum_{i=1}^{n-1} d_iy^i)+ \\
& \ \ +(\sum_{i=1}^{n-1} b_iy^i)(\sum_{i=1}^{n-1} c_ix^i)+(\sum_{i=1}^{n-1} b_iy^i)(\sum_{i=1}^{n-1} d_iy^i)\\
&= x^2(\sum_{i=1}^{n-2} a_{i+1}x^i)(\sum_{i=1}^{n-2} c_{i+1}x^i)+xy(\sum_{i=1}^{n-2} a_{i+1}x^i)(\sum_{i=1}^{n-2} d_{i+1}y^i)+ \\
& \ \ +xy(\sum_{i=1}^{n-2} b_{i+1}y^i)(\sum_{i=1}^{n-2} c_{i+1}x^i)+y^2(\sum_{i=1}^{n-2} b_{i+1}y^i)(\sum_{i=1}^{n-2} d_{i+1}y^i)\\
&= x^2(\sum_{i=1}^{n-2} a_{i+1}x^i)(\sum_{i=1}^{n-2} c_{i+1}x^i)+y^2(\sum_{i=1}^{n-2} b_{i+1}y^i)(\sum_{i=1}^{n-2} d_{i+1}y^i),
\end{split}
\end{equation*}
where the terms involving a factor of $xy$ vanish because $xy=0$. Note that the last equation implies 
\begin{equation*}
\begin{split}
0 &= (\sum_{i=1}^{n-2} a_{i+1}x^i)(\sum_{i=1}^{n-2} c_{i+1}x^i) \\
0 &=(\sum_{i=1}^{n-2} b_{i+1}y^i)(\sum_{i=1}^{n-2} d_{i+1}y^i) \ .
\end{split}
\end{equation*}
From this it is clear that these equations define $\nabla_{\fl_{n-2}}N\times_{k}\nabla_{\fl_{n-2}}N$ provided that $n\geq 3$. Note that the variables $a_{n-1}, \ b_{n-1}, \ c_{n-1},$ and $d_{n-1}$ are free. This gives the result.
\end{proof}
\begin{remark}
From this, we may deduced that the asymptotic defect  $\delta(J_{O}^{\infty}N)$ is $1$. Similar results should be possible for the curve $N(m,l) = \spec{k[x,y]/(x^my^l)}$. Thus, we see that for germs of  reducible curves, there is no reason to expect that, in general, the asymptotic defect will be equal to the embedding dimension of the germ.
\end{remark}

\begin{conjecture}\label{decomp}
Let $C$ be a connected curve\footnote{This means that $C$ is an object of $\sch{k}$ such that $C^{\red}\cong C$ and  $\dim(C)=1$.} which has one singular point $p$. Let $e$ be the degree of natural morphism $\bar C\to C$ where $\bar C$ is the normalization of $C$. Then, for sufficiently large $n$, there exists $P_i(t)\in\bZ[t]$ with $\deg(P_i(t)) \leq 1$ for all $i=1,\ldots,e$ and a fixed $r\in\bN$
such that there is an isomorphism
\begin{equation}
\sA_n(C,p)^{\red} \cong \nabla_{\fl_{P_1(n)}}W\times_{k}\cdots
\times_{k}\nabla_{\fl_{P_e(n)}}W\times_{k}\bA_{k}^{r}
\end{equation}
where $W$ is some connected curve which is analytically isomorphic to $C$ at $O$.
\end{conjecture}

\begin{example}\label{nodecube}
Consider the nodal cubic $Y$ defined by $y^2=x^3+x^2$ and let $O$ be the point at the origin. Then, for all $n\in\bN$, 
$J_O^nY \cong J_O^nN$ where $N$ is the node. This is because $x+1$ is sent to a unit in the coordinate ring of $J_O^nY$. Therefore, for all $n\in\bN$,
$\sA_n(Y,O) \cong \sA_n(N,O)$. Thus, the conjecture above is verified in the case of $Y$ as well by setting $W$ equal to $N$ in the above.
\end{example}

\section{Auto Igusa-zeta series of a curve with a singular point}

Theorem \ref{cusp} immediately implies the following formula for the reduced auto Igusa-zeta function of $C$ at the origin $O$:
\begin{equation*}
\bar\zeta_{C,O}(t) = 1+t+t^2+t^3\sum_{n=1}^{\infty}[\nabla_{\fl_{2n}}C]\bL^{-2n}t^n ,
\end{equation*}
where the coefficients of the first three terms in the summation were calculated in \S\ref{compsec}. Here, we must assume that $\mbox{char}(k)\neq 2, 3$. Thus, by performing the substitution $t=s^2$ and subtracting the first three terms, we have that 
\begin{equation*}
\bar\zeta_{C,O}(s^2) - (1+s^2+s^2) = s^6\sum_{n=1}^{\infty}[\nabla_{\fl_{2n}}C]\bL^{-2n}s^{2n}.
\end{equation*}
After inverting $s^6$, the right hand side is precisely the even terms of the reduced motivic Igusa-zeta series of $C$ along $J_{O}^{\infty}\bA_{\kappa}^{n}$. We denote this power series, which we will define below, by $\Theta_{C,\fl}^{\star}(t)$ where $\star$ is some subset of $\bN$. Thus, we may rewrite the previous formula in the following way:
\begin{equation*}
\bar\zeta_{C,O}(s^2) - (1+s^2+s^4)  = s^6\cdot\Theta_{C,\fl}^{\star_{\mathbf{2}}}(s)
\end{equation*}
where $\star_n$ denotes the subset determined elements of $\bN$ divisible by $n$. One can then show that  $\Theta_{C,\fl}^{\star_{\mathbf{2}}}(s)$ is, at the very least,  an element of $\sG_{k}(t)$. In fact, 
using Example 2.4 of \cite{Ve}, we have that 
\begin{equation*}
\Theta_{C,\fl}(s)=\frac{\bL+(\bL-1)s+(\bL^2-\bL)s^5+\bL^2s^6}{(1-\bL s^6)(1-s)} 
\end{equation*}
Thus, we may easily collect all even terms and obtain
\begin{equation*}
\Theta_{C,\fl}^{\star_{\mathbf{2}}}(s)=\frac{\bL+(\bL-1)s^2+(2\bL^2-\bL)s^6}{(1-\bL s^6)(1-s^2)}
\end{equation*}
Thus, we arrive at 
\begin{equation*}
\bar\zeta_{C,O}(t)=\frac{1-(\bL+1)t^3+\bL t^4 +(\bL-1)t^5 +2\bL^2t^6}{(1-\bL t^3)(1-t)}
\end{equation*}

Therefore, in the end, we have that $\bar\zeta_{C,O}(t)$ will be an element of  $\sG_k(t)$.

\begin{definition}
Let $X$ and $Y$ be objects of $\sch{k}$ and let $p$ be a point of $Y$. We define the {\it motivic Igusa-zeta series} of $X$ along $J_p^{\infty}Y$ at $p$ to be the power series 
\begin{equation*} \label{Ig0}
\mbox{Igu}_{X,J_p^{\infty}Y}(t) = \sum_{n=0}^{\infty}[\nabla_{J_p^{n+1}Y}(X\times_{k}\kappa(p))]\bL^{-dim{}_p(X)\cdot\ell(J_p^{n+1}Y)}t^n \in\sH_{\kappa(p)}[[t]].
\end{equation*} \label{Ig1}
 We define the {\it reduced motivic Igusa-zeta series} of $X$ along $J_p^{\infty}Y$ to be the power series 
 \begin{equation*}
 \Theta_{X, J_p^{\infty}Y}(t) = \sigma_{\kappa(p)}^{\prime}(\mbox{Igu}_{X,J_p^{\infty}Y}(t)).
 \end{equation*}
\end{definition}

\begin{remark}
Note that the series introduced in Equation \ref{Ig0} was originally introduced at the very beginning of \S9 of \cite{Sch2}.
Note also that for any ring $R$, any subset $\star$ of $\bN$ and any power series $P(t) \in R[[t]]$, we always denote by $P^{\star}(t)$ the element of $R[[t]]$ determined by the formal summation of all terms $a_it^i$ of $P(t)$  such that $a_i\in R$ and $i\in\star$. Clearly then, $P(t) = P^{\star_1}(t)$.
\end{remark}

\begin{example}
Consider the case of the reduced auto Igusa-zeta function of the node $N$ at the origin $O$. A quick calculation using Theorem \ref{node} and \S\ref{compsec} yields
\begin{equation*}
\begin{split}
\bar\zeta_{N,O}(t) &= 1+t\cdot\sum_{n=1}^{\infty}[\nabla_{\fl_n}(N^2)]\bL^{-2n}t^n\\
&=1+t\cdot\Theta_{N^2, \fl}(t),
\end{split}
\end{equation*}
where we use the short hand $X^m = X\times_k\cdots\times_kX$ ($m$-times fiber product). 
Note that for any $X,\  Y\in\sch{k}$ and any $\fn\in\fatpoints{k}$, 
$\nabla_{\fn}(X\times_kY)\cong (\nabla_{\fn}X)\times_k(\nabla_{\fn}Y)$.
By \cite{DL1}, we know that $\Theta_{N^2, J_P^{\infty}\bA_k^1}(t)$ to be an element of $\sG_k(t)$. Thus, $\bar\zeta_{N,O}$ is rational. More explicitly, just as in the case of the cusp, we may use \cite{Ve}, to obtain
$[\nabla_{\fl_{n+1}}N] = (n+2)\bL^{n+1} - (n+1)\bL^n$
Thus, 
\begin{equation*}
[\nabla_{\fl_{n+1}}N^2] = ([\nabla_{\fl_{n+1}}N])^2= ((n+1)^2\bL^2 - 2(n+2)(n+1)\bL + (n+1)^2)\bL^{2n} .
\end{equation*}
Making the substitution $s = \bL^2t$, we arrive at 
\begin{equation*}
\begin{split}
\Theta_{N^2,O}(t) &= \bL^2\sum_{n=0}^{\infty}(n+2)^2s^n -2\bL\sum_{n=0}^{\infty}(n+2)(n+1)s^n + \sum_{n=0}^{\infty}(n+1)^2s^n \\
&= \bL\frac{2-s+s^2}{(1-s)^3} - 2\bL\frac{2}{(1-s)^3} +\frac{3-s}{(1-s)^3} \\
&= \frac{(2\bL^2 -4\bL + 3) -\bL^2(\bL^2+1)t + \bL^4 t^2}{(1-\bL^2 t)^3}
\end{split}
\end{equation*}
Therefore, we arrive at the following rational expression for the auto Igusa-zeta series of the node at the origin:
\begin{equation*}
\bar\zeta_{N,O}(t) = \frac{1 - (\bL^2 +4\bL -3)t+\bL^2(2\bL^2-1)t^2 - \bL^4(3\bL^2-1)t^3}{(1-\bL^2t)^3}
\end{equation*}
\end{example}

\begin{conjecture}\label{hope}
Let $C$ be a connected curve which has only one singular point $p$ and  consider the normalization morphism $f:\bar C\to C$. Then, there exists $r, b, q\in \bN$ such that
\begin{equation*}
\cdot\bar\zeta_{C,p}(t^r) = \frac{1-t^{r(b-1)}}{1-t^r} + t^{rb}\cdot\Theta_{W^{\deg(f)},\fl}^{\star_q}(t)
\end{equation*}
as elements of $\sG_k[[t]]$, where $W$ is some connected curve which is analytically isomorphic to $C$ at $O$.
\end{conjecture}

\begin{example}
Consider the nodal cubic $Y$ defined by $y^2=x^3+x^2$ and let $O$ be the point at the origin. Then, by Example \ref{nodecube}, the conjecture above is verified for  $Y$ by letting $W$ be equal to  $\spec{k[x,y]/(xy)}$ in the above.
\end{example}

\begin{remark}
We may further postulate that $\Theta_{C^e,\fl}^{\star_q}(t)$ is an element of $\sG_k(t)$ for any $e, q\in \bN$ and for any curve $C\in\Var{k}.$ This conjectural statement together with the previous conjecture will prove that $ \bar\zeta_{C,p}(t) \in \sG_k(t)$ whenever $C$ is an irreducible curve and $p$ is the only singular point of $C$. 
\end{remark}

Note that 
\begin{equation*}
\bL^{d}\cdot\bar\zeta_{\bA_{k}^d,p}(t) = \Theta_{\bA_{k}^d, \fl}(t).
\end{equation*}
This equation together with Proposition \ref{smzeta} and Theorem 9.1 of \cite{Sch2} immediately prove the following proposition.

\begin{proposition}
 Let $X$ be an object of $\sch{k}$ which is smooth at $p\in X$. We have the following identity:
$$[X]\cdot\bar\zeta_{X,p}(t)= \Theta_{X, \fl}(t).$$
\end{proposition}

\begin{remark}
Thus, whenever Conjecture \ref{hope} does hold (such as in the case of the cuspidal cubic), we may regard the result as a generalization of the previous proposition.
\end{remark}

\section{Motivic integration via generating series}\label{mot}

 Fundamentally, the material of the previous section should be about the relationship between two potential types of motivic integrals. In this section, we investigate this relationship, but first we must answer the following question.

\begin{question}\label{tempq}
Let $X$ be an object of $\sch{k}$ and let $p$ be a point of $X$.  When is there morphism of varieties $\rho_{n-1}^{n}:\sA_{n}(X,p)^{\red} \to \sA_{n-1}(X,p)^{\red}$? Moreover, when does such a morphism arise in a natural way?
\end{question}

Clearly, if $X \in \sch{k}$ is smooth at $p$, then, by Theorem \ref{them}, there exists a morphism $\rho_{n-1}^{n}$ given by projection. 
Moreover, in the case of the cuspidal cubic $C$ (resp. the node $N$), the morphism $\rho_{n-1}^{n}$ is by the truncation $\nabla_{\fl_{2n}}C\to \nabla_{\fl_{2n-2}}C$ (resp. 
$\nabla_{\fl_{n}}N^2 \to \nabla_{\fl_{n-1}}N^2$). The following lemma shows that in fact we always have a natural morphism $\rho_n$, giving a positive answer to Question \ref{tempq}.

\begin{lemma}\label{yon}
Let $X \in \sch{k}$ and let $p$ be a point of $X$. Then, for all $n \in \bN$, there is a natural morphism
$\rho_{n-1}^{n}:\sA_n(X,p)^{\red}\to \sA_{n-1}(X,p)^{\red}$.
\end{lemma}

\begin{proof}
By the Yoneda lemma, it is enough to show that there is a canonical set map from $\sA_n(X,p)^{\red}(F) \to \sA_{n-1}(X,p)^{\red}(F)$ where $F$ is any field extension of $k$.  This amounts to showing that there is commutative diagram
\[\begin{CD}
A/\fm^n\>{f}>> A/\fm^n\otimes_k F \\
\V{c}VV \V{c\otimes\phi}VV \\
A/\fm^{n-1}\>{\bar f}>> A/\fm^{n-1}\otimes_k F
\end{CD}\]
where $A$ is a local ring containing $k$ with maximal ideal $\fm$, $\phi$ is an automorphism of $F$, $c$ is the canonical surjection, and where $\bar f$ is 
induced by $f$. Indeed, $\bar f$ exists since $f(\fm^{n-1}/\fm^n) \subset \fm^{n-1}/\fm^n \cdot F$ for any ring homomorphism $f: A/\fm^n\to A/\fm^n\otimes_k F$.
\end{proof}

With this in mind, our approach in connecting the previous material to motivic integration is to ask
questions about lifts to $Y_{n+1}\to \sA_{n+1}(X,p)$ of a given smooth morphism $Y_n \to \sA_{n}(X,p)$ in $\sch{k}$.

\begin{lemma} \label{lifting}
Let $X$ be any  object of $\sch{k}$ and let $p\in X$ with residue field $\kappa(p)$.
Let $Y_n\in\sch{\kappa(p)}$ and suppose that $Y_n$ is affine.  Assume that there exists a smooth morphism $f: Y_n \to J_p^{n}X$. Then, there exists a unique smooth morphism $\bar f: Y_{n+1} \to J_p^{n+1}X$  where $Y_{n+1} \in \sch{\kappa(p)}$ such that $Y_n \cong Y_{n+1} \times_{J_p^{n+1}X}J_p^{n}X.$
\end{lemma}

\begin{proof} It is enough to prove this statement for any two fat points $\fn, \fm \in\fatpoints{\kappa(p)}$ admitting a closed immersion $\fn\inj\fm$ under the assumption that there exists a smooth morphism $Z \to \fn$ where $Z$ is some affine scheme. Further, we may reduce to the case where the closed immersion $\fn
\inj \fm$ is given by a square zero ideal $J$. Then, it is well-known, see for example Theorem 10.1 of \cite{Ha2}, that the
obstruction to lifting smoothly to $Z' \to \fm$ is an element of $H^{2}(Z, T_{Z}
\otimes \tilde J)$ where $T_Z$ is the tangent bundle of $Z$. Since
$T_{Z}\otimes \tilde J$ is quasi-coherent and $Z$ is assumed to be affine, we
have that 
\begin{equation*}
H^{2}(Z, T_{Z} \otimes \tilde J) = 0 \ ,
\end{equation*}
by Theorem 3.5 of Chapter III of \cite{Ha1}. The uniqueness part quickly follows as the obstruction to uniqueness is an element of $H^{1}(Z, T_{Z} \otimes \tilde J)$, again by Theorem 10.1 of \cite{Ha2}, which is also trivial since $Z$ is affine and $T_{Z} \otimes \tilde J$ is quasi-coherent.
\end{proof}

\begin{remark} Note that $Y_n$ being affine here is important; otherwise, there is a cocycle condition on $f$ that must be satisfied in order to insure that there is such a lift -- i.e., to insure that the morphism we would obtain by gluing is smooth. 
\end{remark}

Let $X$ be an object of $\sch{k}$. Let $Y_n\in\sch{\kappa(p)}$ be affine of pure dimension $d$.
Assume that $Y_n \to J_p^{n}X$ is a  smooth morphisms. Then, there is an affine scheme $Y_{n+1}$ equipped with a smooth morphism $Y_{n+1}\to J_p^{n+1}X$ such 
that  $Y_{n} \cong Y_{n+1} \times_{J_p^{n+1}X}J_p^{n}X$ by the previous lemma. By Lemma \ref{yon}, we can show that there is natural morphism 
$$\pi_{n}^{n+1}:(\nabla_{J_p^{n+1}X}Y_{n+1})^{red} \to (\nabla_{J_p^{n}X}Y_n)^{red}. $$ 

Indeed, we may cover $Y$ by a finite number of opens
$U$, each of which will admit an \'{e}tale morphism $U \to \bA_{J_p^{n+1}X}^{d}$ where $d=\dim(Y_{n+1})=\dim(Y_n)$ . As \'{e}tale morphisms are stable under base change, the restriction $U' \to J_p^{n}X$ of $U$ also admits an \'{e}tale morphism $U' \to \bA_{J_p^{n+1}X}^{d}$. Therefore, from the start, we may assume
that we have \'{e}tale morphisms $Y_{n+1} \to 
\bA_{J_p^{n+1}X}^{d}$ and $Y_n\to\bA_{J_p^{n}X}^{d}$. For notational reasons, let $J(n)= J_p^{n}X$ in the following. We then have the following isomorphisms:
\begin{equation*}
\begin{split}
\nabla_{J(n+1)}Y_{n+1}&\cong Y_{n+1} \times_{\bA_{J(n+1)}^{d}}\nabla_{J(n+1)}\bA_{J(n+1)}^{d} \\  
\nabla_{J(n)}Y_n&\cong Y_n
\times_{\bA_{J(n)}^{d}}\nabla_{J(n)}\bA_{J(n)}^{d}  .
\end{split}
\end{equation*}
Using Lemma \ref{algebralemma}, the morphism $\rho_{n-1}^{n}$ given to us Lemma \ref{yon} induces a commutative diagram
\[\begin{CD}
(\nabla_{J(n+1)}Y_{n+1})^{\red}\>{\cong}>> Y_{0}\times_{\kappa(p)}\sA_{n+1}(X,p)^{\red}\times_{\kappa(p)}
\bA_{{\kappa(p)}}^{d(\ell(J(n+1))-1)} \\
\V{\pi_{n}^{n+1}}VV \V{}VV \\
(\nabla_{J(n)}Y_n)^{\red}\>{\cong}>> Y_0 \times_{\kappa(p)}\sA_n(X,p)^{\red}\times_{\kappa(p)}
\bA_{{\kappa(p)}}^{d(\ell(J(n))-1)}
\end{CD}\]
where $Y_0 \cong(Y_n)^{\red}\cong (Y_{n+1})^{\red}$. Here the morphism in the downward direction on the right is induced by an automorphism of $Y_0$, $\rho_{n}^{n+1}$ on the middle factor, and projection of the first $d(\ell(J(n))-1)$ coordinates of $\bA_{{\kappa(p)}}^{d(\ell(J(n+1))-1)}$ onto $\bA_{{\kappa(p)}}^{d(\ell(J(n))-1)}$.
Thus, we arrive at a locally ringed space $\sA$ defined by 
\begin{equation*}
\sA := \varprojlim_{n} (\nabla_{J(n)}Y_n)^{\red}
\end{equation*}
along with morphisms $\pi_n : \sA\to (\nabla_{J(n)}Y_n)^{\red}$. We call $\sA$ {\it the infinite auto-arc space of} $Y_n$ {\it along the germ} $(X,p)$, and we will sometimes denote it by $\sA_{X,p}(Y_n)$ or just by $\sA$. 

\begin{lemma}
The locally ringed space $\sA_{X,p}(Y_n)$ constructed above is a scheme. 
\end{lemma}

\begin{proof}
This follows from the fact that the morphisms $\pi_{n-1}^{n}$ are affine.
\end{proof}

\begin{remark}
If we define, $\sA_{\infty}(X,p)$ to be 
$\nabla_{J_p^{\infty}X}J_p^{\infty}X$. Then, it follows that 
$\sA_{\infty}(X,p)^{\red}\cong \varprojlim \sA_n(X,p)^{\red}$. Moreover, it also follows that $$\sA_{X,p}(Y_n)\cong Y_0 \times_k \sA_{\infty}(X,p)^{\red}\times_k \bA_{k}^{\infty}$$
where $Y_0 \cong (Y_n)^{\red}$.
\end{remark}

One type of natural motivic volume one can introduce on $\sA:=\sA_{X,p}(Y_n)$ is 
\begin{equation}
\nu_{X,p}^{\auto}(\sA,n):=[(\nabla_{J(n)}Y_{n}^{\red}]\bL^{-d_n}
\end{equation}
where $d_n= (\dim(Y_0)+\dim{}_p(X)-1)\ell(J(n))+n$. Then, we define {\it the motivic integral along the length function} to be 
\begin{equation}
\int_{\sA}\bL^{-\ell}d\nu_{X,p}^{\auto} : = \sum_{n=0}^{\infty}\nu_{X,p}^{\auto}(\sA,n+1)\bL^{-\ell(J(n+1))},
\end{equation}
whenever the right hand side converges in $\hat\sG_{\kappa(p)}$.
Thus, in summary, we have the following theorem.
\begin{theorem}
Let $X$ be and object of $\sch{k}$. Let $Y_n\in\sch{k}$ be an affine scheme of pure dimension $d$ which admits a smooth morphism $Y_n \to J_p^nX$ for some $p\in X$. Let $\sA$ be the infinite auto-arc space of $Y_n$ along $(X,p)$. Then, there is a motivic volume $\nu_{X,p}^{\auto}(\sA,n)\in\hat\sG_{k}$ at level $n$ for each $n\in \bN$ such that
\begin{equation}
\int_{\sA}\bL^{-\ell}d\nu_{X,p}^{\auto} = [Y_0]\bL^{-\dim(Y_0)}\cdot \bar\zeta_{X,p}(\bL^{-1})
\end{equation}
in some ring extension $R_{\kappa(p)}$ of $\bar\sG_{\kappa(p)}$,
where $d=\dim(Y_0)$ and $Y_0\cong (Y_n)^{\red}$.
Moreover, 
\begin{equation*}
\nu_{X,p}^{\auto}(\sA) := \varprojlim_{n\in\bN}\nu_{X,p}^{\auto}(\sA,n)\bL^{-\ell(J_p^nX)+n}
\end{equation*}
exists as an element of $R_{\kappa(p)}.$
\end{theorem}

\begin{question}
What is the ring $R_{\kappa(p)}$ and does it have a simple 
description? Moreover, what is the minimal ring extension 
$R_{\kappa(p)}$ such that 
$\int_{\sA}\bL^{-\ell}\nu_{X,p}^{\auto}\in R_{\kappa(p)}$ for 
some class of varieties $X$ and some specified class of singularities 
$p\in X$?
\end{question}

Our inability to answer the previous question sufficiently is a barrier to this approach.  However, this theorem does give us the following two corollaries.

\begin{corollary}
Assume that $X$ is smooth of pure dimension $d$ and let $p$ be an arbitrary point of $X$. Moreover, let $\sA$ be the infinite auto-arc space over $(X,p)$, $\nu^{\op{auto}} :=\nu_{X,p}^{\op{auto}}$, and $\mu^{\op{mot}}$ the standard geometric motivic volume on $\nabla_{\fl}X$, then
\begin{equation*}
\int_{\sA}\bL^{-\ell}d\nu^{\op{auto}} = \mu^{\op{mot}}(\nabla_{\fl}X)\Theta_{X,\fl}(\bL^{-1}) = \frac{\mu^{\op{mot}}(\nabla_{\fl}X)}{1-\bL^{-1}} = \frac{[X]\bL^{-d}}{1-\bL^{-1}}
\end{equation*}
\end{corollary}

\begin{corollary}
Let $X$ be an object of $\sch{k}$ and let $p$ be a point of $X$.  Assume that $\zeta_{X,p}(t)\in \sG_k(t)$. Let $Y_n\in\sch{k}$ be an affine scheme of pure dimension $d$ and assume that $Y_n$ is smooth over $J_p^{n}X$ for some $n\in\bN$. Then, there exists a finite ring extension $R_{\kappa(p)}$ of $\bar\sG_{\kappa(p)}$ such that $\int_{\sA_{X,p}(Y_n)}\bL^{-\ell}d\nu_{X,p}^{\auto} \in R_{\kappa(p)}$ and such that $\nu_{X,p}^{\auto}(\sA)\in R_{\kappa(p)}$.
\end{corollary} 

Although, we do not have a proof of the rationality of $\bar\zeta_{X,p}(t)$ when $X$ is not smooth at $p$ (except in the case of the cusp and the node), we may consider the following adjustment.

\begin{definition}
Let $X$ be and object of $\sch{k}$. Let $Y_n\in\sch{k}$ be an affine scheme, of pure dimension $d$, admitting a smooth morphism $Y_n \to J_p^nX$ for some $p\in X$. Let $\sA$ be the associated infinite auto arc space of $Y_n$ along $(X,p)$. We define the {\it adjusted motivic volume of $\sA$} with respect to  $(X,p)$ at level $n$ to be  
\begin{equation}
\mu_{X,p}^{\auto}(\sA,n) := [\pi_{n}(\sA)]\bL^{-d_n}
\end{equation}
 when it exists an element of $\hat\sG_{k}$. In the above, $d_n= (\dim(Y_0)+\dim{}_p(X)-1)\ell(J_p^nX)+n$ and $Y_0\cong(Y_n)^{\red}$. As before, we define {\it the motivic integral along the length function} to be
\begin{equation}
\int_{\sA}\bL^{-\ell}d\mu_{X,p}^{\auto} := \sum_{n=1}^{\infty}\mu_{X,p}^{\auto}(\sA,n+1)\bL^{-\ell(J_p^{n+1}X)}.
\end{equation}
whenever the right hand side converges. Finally, if $\mu_{X,p}^{\auto}(\sA,n)$ exists for all $n\in\bN$, then we define
the {\it adjusted motivic volume of} $\sA$ with respect to $(X,p)$ to be
\begin{equation}
\mu_{X,p}^{\auto}(\sA) := \varprojlim_{n\in\bN}\mu_{X,p}^{\auto}(\sA,n)\bL^{-\ell(J_p^nX)+n}
\end{equation}
as an element of $\hat\sG_{\kappa(p)}$.
\end{definition}  

\begin{theorem}
Let $C$ be a curve with only one singular point $p$. Assume further that Conjecture \ref{decomp} holds for $C$. Let $Y_n\in\sch{k}$ be an affine scheme admitting a smooth morphism $Y_n \to J_p^nC$ and let $\sA$ be the infinite auto-arc space of $Y_n$ along $(C,p)$. Then, 
the adjusted motivic volume $\mu_{C,p}^{\auto}(\sA)$  exists as an element of $\hat\sG_{\kappa(p)}$.
\end{theorem}

\begin{proof}
This follows from the fact that $\sA$ will be definable in the language of Denef-Pas and from the fact that $\mu_{C,p}^{\auto}(\sA)$ is just the classical (geometric) motivic measure of $\sA$ in this case.
\end{proof}

Thus, under the conditions of the previous theorem, it immediately follows, by the same argument as can be found in the proof of Theorem 5.4 of \cite{DL1}, that the {\it auto Poincar\'{e} series} $P_{\sA}^{\auto}(t)$ defined by 
\begin{equation}
P_{\sA}^{\auto}(t):= \sum_{n=0}^{\infty}[\pi_{n+1}(\sA)]\bL^{-(\dim(Y_0)+\dim{}_p(X))\ell(J_p^{n+1}X)}t^n
\end{equation}
is rational. Moreover, by Theorem 5.4 of \cite{DL1}, we have
\begin{equation}
\int_{\sA}\bL^{-\ell}d\mu_{C,p}^{\auto} = P_{\sA}^{\auto}(\bL^{-1})=[Y_0]\bL^{-\dim(Y_0)}\cdot P_{\sA_{C,p}(\spec{k})}(\bL^{-1})\in \bar\sG_{k}[(\frac{1}{1-\bL^{-i}})_{i\in\bN}] 
\end{equation}
as $P_{\sA}^{\auto}(t) = f(t)/g(t)$ where $f(t),g(t)\in\sG_k$ and where $g(t)$ is a product of elements of the form $\bL^j -1$ and of the form $1-\bL^{-i}t^b$ where $b,i,j\in \bN$ ($b,j\neq 0$), see Theorem 5.1 of loc. cit. Finally, by the Corollary of Theorem 5.1 of loc. cit., 
$$\mu_{C,p}^{\auto}(\sA)\in\bar\sG_{k}[(\frac{1}{1-\bL^{-i}})_{i\in\bN}]$$
where $\bar\sG_{k}$ is the image of $\sG_k$ in $\hat\sG_k$.

\begin{question}
Is it possible that Conjecture \ref{decomp} can be extended to include higher dimensional varieties in such a way that similar results may be obtained for more general types of germs $(X,p)$? In other words, to what extent exactly will the motivic volume $\mu_{X,p}^{\auto}(\sA)$ be the same as the classical (geometric) motivic volume of $\sA$?
\end{question}

\begin{example} If $F$ is a field extension of $k$ and $Y_n$ is smooth over $J_{(0)}^{n}\spec{F}=\spec{F}$, then $Y_n \cong Y_0$ for all $n$. Therefore, in this case, $\sA = Y_0$. Thus, 
\begin{equation*}
\begin{split}
\int_{Y_0} \bL^{-\ell}d\nu_{\spec{F},(0)}^{\auto} &=  [Y_0]\bL^{-\dim(Y_0)}\cdot\frac{1}{1-\bL^{-1}} \\
\nu_{\spec{F},(0)}^{\auto}(Y_0) &= [Y_0]\bL^{-\dim(Y_0)}  .
\end{split}
\end{equation*}
Similar results can easily be obtained for other zero-dimensional schemes $X\in\sch{F}$.
\end{example}

\section{Appendix A: Smooth reductions}

In this section, we unravel what it means for $Y_n\to J_p^{n}X$ to be a smooth morphism when $Y_n\in\sch{\kappa(p)}$, $X\in\sch{k}$, and $p\in X$. 
We know that in general smoothness does not descend via a faithfully
flat morphism; however,  we have the following:
\begin{proposition} \label{egaprop}
Let $f : X \to Y$ and $h : Y' \to Y$ be two morphisms in $\sch{\kappa}$.
Let $X' =X\times_Y Y'$ and let $f' : X' \to Y'$ be the canonical projection. Suppose
further that $h$ is quasi-compact and faithfully flat, then $f$ is smooth if and
only if $f'$ is smooth. \end{proposition} 
\begin{proof}
This is a special case of Proposition 6.8.3 of \cite{G}. \end{proof}

 Let $X \in \sch{\kappa}$ be affine and write $X = \Spec A$. Choose a minimal system
of generators $g_1,\ldots, g_s$ of the nilradical $nil(A)$ of $A$. Let $x_1,
\ldots, x_s$ be $s$ variables and let $J$ be the kernel of the map from
${\kappa}[x_1,\ldots,x_s]$ to $A$ which sends $x_i$ to $g_i$. We set $R :=
{\kappa}[x_1,\ldots, x_s]/J$. Then, $R\inj A$. Here, $R$ is nothing other than the {\it maximum artinian subring of} $A$. We have the following:

\begin{lemma} \label{algebralemma}
Let $X=\spec A$ be a connected affine scheme in $\sch{\kappa}$, set $\fn = \Spec
R$ where $R$ is the maximum artinian subring of $A$, and let $l$ be any positive
integer. Then, $\fn$ is fat point over ${\kappa}$, and we have the following
decompositions:

\noindent (a) \quad $X^{\red} \cong X\times_\fn \Spec {\kappa}$ 

\smallskip

\noindent (b) \quad $X \times_{\bA_{\fn}^{d}}\bA_{{\kappa}}^{dl} \cong X^{red}\times_{\kappa}
\bA_{{\kappa}}^{d(l-1)}$.
\end{lemma}
\begin{proof}
Write $X = \Spec A$ for some finitely generated ${\kappa}$-algebra. It is basic that $R
\inj A$. Let $\sM = (x_1,\ldots, x_s)R$. Clearly, 
$\sM$ is a maximal ideal of $R$.  
Moreover, $\sM\cdot A \subset nil(A)$ by
construction. Therefore, there exists an $N$ such that 
$\sM^N = 0$. Thus, $R$ is artinian ring with residue field $\kappa$.
We assumed $X$
was connected so that $R$ would be local. 
Indeed, by injectivity of $R \inj A$, any direct sum decomposition of $R$ would
immediately imply a direct sum decomposition of $A$ 
as it would entail that $R$ (and hence $A$) contains orthogonal idempotents $e_1
\neq e_2$.  

Note that the containment $\sM\cdot A \subset nil(A)$ is actually an equality by
construction. Now, use the fact that ${\kappa} = R/\sM $ so 
that $$A \otimes_R {\kappa} \cong A\otimes_R (R/\sM) \cong (A/\sM A) \otimes_R R \cong
A/\sM A \cong A/nil(A)$$ where the second 
isomorphism is a well-known property of tensor products for $R$-algebras. This
proves part (a). 

Part (b) is really a restatement of the work done in the preceding paragraph. 
One should just note that $$ X\times_{\bA_{\fn}^{d}} \bA_{{\kappa}}^{dl} \cong
X\times_{\bA_{\fn}^{d}} \spec{\kappa} \times_{\kappa} \bA_{{\kappa}}^{dl} $$
so that we can apply (a) to the right hand side to obtain 
$$X\times_{\bA_{\fn}^{d}} {\kappa} \times_{\kappa} \bA_{{\kappa}}^{dl} \cong X^{\red}
\times_{\bA_{{\kappa}}^{d}}\bA_{{\kappa}}^{dl} \ .$$
This proves the result part (b). 
\end{proof}

\begin{theorem} Let $X=\spec A$ be connected. Then,  $X^{\red}$ is smooth if
and only if there exists a smooth morphism $X \to \fn$ where $\fn = \Spec R$ such that $R$ is the maximum artinian subring of $A$.
\end{theorem}

\begin{proof}
This is just a restatement of Proposition \ref{egaprop} where $Y' = \Spec {\kappa}$, $Y
= \fn$,
and $Y' \to Y$ is the canonical morphism. Indeed, by Lemma \ref{algebralemma},
$X' := X\times_Y
Y' \cong X^{\red}$, and the homomorphism of rings $R \to {\kappa} $ given by modding out
by $\sM$ is both surjective and flat.
\end{proof}

In summary, we have proven the following theorem.

\begin{theorem} Let $Y_n\in\sch{\kappa}$ be affine and let $Y_0:=(Y_n)^{\red}$. Then, the following three conditions are equivalent.
\begin{enumerate}
 \item $Y_0$ is smooth over $\kappa$ and the maximal artinian subring of $\sO_{Y_n}(Y_n)$ is the coordinate ring of $J_p^nX$.
  \item $Y_n$ is the trivial deformation of $Y_0$ over $J_p^nX$.
  \item There is a smooth morphism $Y_n\to J_p^nX$. 
  \end{enumerate}
  \end{theorem}
  
  This equivalence is pertinent to the situation described in \S \ref{mot} in that it clarifies what exactly $Y_n$ looks like. Moreover, it shows that $$\sA_{X,p}(Y_n) \cong \sA_{X,p}(Y_0)$$ for all $n\in\bN$. 

Note that the condition of $Y_n$ being affine is not as restrictive as it may appear because we may work locally. In other words, infinite auto-arc space may be defined locally and the material of \S \ref{mot} will extend to non-affine schemes $Y$ such that $Y^{\red}$ is smooth. To what degree the condition on the smoothness of $Y^{\red}$ can be relaxed in \S\ref{mot} appears to me as an interesting question.

\section{Appendix B: Auto-arc spaces with field automorphisms}

We briefly discuss the possibility of including automorphism of the field. Namely, we always considered $J_p^nX$ as an object of $\sch{\kappa(p)}$ where $X\in\sch{k}$ and $p\in X$ with residue field $\kappa(p)$ and that $\nabla_{J_p^nX}(-)$ as an endofunctor on $\sch{\kappa(p)}$. However, if $p\in X$ is a closed point and we allow, for the moment,
$J_p^nX$ as an object of $\sch{k}$ and let $F$ be the functor determined by considering $\nabla_{J_p^nX}(-)$ as an endofunctor on $\sch{k}$. Then, 
$$F(J_p^nX) = \sA_{n}(X,p)\times_k\mbox{Aut}_{k}(\spec{\kappa(p)}). $$
Furthermore, in this case, $F(J_p^nX)$ will also be an object of $\sch{k}$ --i.e., $\mbox{Aut}_{k}(\spec{\kappa(p)})$ will be of finite type of $k$. Thus, 
$F(J_p^nX)$ will have a well-defined class in $\grot{\Form{k}}$, and we can run through the 
definitions and results of this paper in this case. The main problem I see with this 
approach is that it might be possible that $[\mbox{Aut}_{k}
(\spec{\kappa(p)})]$ is zero-divisor. Thus, the morphism of monoids 
$M:\sG_{\kappa(p)}\to \sG_{k}$ defined by sending $[X]\bL^i$ to $[X\times_{\kappa(p)}k]\cdot [\mbox{Aut}_{k}(\spec{\kappa(p)})]$ may not be injective. Thus, although we may look at the auto Igusa-zeta series, for example, in this new context via multiplication $\zeta_{X,p}^{\auto}(t) \mapsto M(\zeta_{X,p}^{\auto}(t))$, it is not at all clear that one can recover rationality results about $\zeta_{X,p}^{\auto}(t)$, for example, by studying auto-arc spaces\footnote{Although, I am inclined to believe that there are still interesting things that one can say in this new yet specialized context.} in this new sense. Therefore, we are pushed to ask the following question.

\begin{question}
Let $k$ be a field and let $F$ be a finite field extension of $k$. Will it ever be the case that the class of $\mbox{Aut}_k(F)$ in $\grot{\Form{k}}$ is a  zero-divisor? Likewise, will it ever be the case that  $\sigma_k([\mbox{Aut}_k(F)])$ will be a zero-divisor of $\grot{\Var{k}}$?
\end{question}

\section{Appendix C: Sage script for computing affine arc spaces}\label{appC}

In this section, I provide my code, written in Sage 6.2.Beta1 (cf., \cite{S} with needed interface with Singular \cite{DGPS}) and Python 2.7.6 (cf., \cite{P}, which will need NumPy \cite{NP} installed), which computes the arc space of an affine scheme $X$ with respect to a fat point $\fn$ in characteristic $0$. Note that the running time increases substantially when the length of the fat point $\ell(\fn)$ increases even modestly, and it also increases dramatically when the fat point $\fn$ has small length but the affine scheme $X$ is even modestly complicated. I am not sure exactly how to quantify the computational complexity here, but that is an interesting question. It looks like computations of arc spaces are destined to be slow. For example, using the SageMathCloud (available at https://cloud.sagemath.com), it took two hours to compute the auto-arc $\sA_8(N,O)$ of the node $N$ at the origin $O$.

I have decided not to include in the code how to compute the reduced arc space. Thus, this must be done by hand (which can be extremely tedious) or done using Sage at the terminal by the user. Likewise, I have not taken up the matter of computing the arc space in positive characteristic. Although, I am more or less certain that this can be done without issue in Sage. Finally, the output is not great and could be organized in better ways, but this question I leave to the user. It does produce the ideal of definition of the arc space which is enough for my purposes.

\medskip

\begin{lstlisting}
import sys
import datetime
import operator
from sage.symbolic.expression_conversions import PolynomialConverter

## ########################################################
#
# Sage code for computing arc spaces
#
## ########################################################

## ########################################################
## Class to organize methods and storing data variables
## ########################################################
class Space:

    def __init__(self):
        self.numvars = 0
        self.numeqs  = 0
        self.firstequation = 0
        self.fatvars = 0
        self.fateqs  = 0
        self.firstfatequation = 0
        return

    def setEquations(self):
        print("Creating functions for your space...")
        return

    def setFatEquations(self):
        print("Creating functions for your fat point...")
        return

    def toString(self):
        msg = "Symbols: " + str(self.numvars) + "\t"
        msg = msg + "Equations: " + str(self.numeqs) + "\n"
        return msg

    def toFatString(self):
        msg = "Symbols: " + str(self.fatvars) + "\t"
        msg = msg + "Equations: " + str(self.fateqs) + "\n"
        return msg

## ########################################################
## Helper methods
## ########################################################
def getInt(msg):
    my_input = raw_input(msg)
    try:
        return int(my_input)

    except:
        print("Input should be an integer, please try again")
        return getInt(msg)
## ########################################################
def debug(msg):
    now = datetime.datetime.now()
    msg = "[" + str(now) + "] " + str(msg)
    print(msg)
    return

## ########################################################
## Begin main program
## ########################################################
if __name__ == '__main__':


    mySpace = Space()
    mySpace.numvars = getInt("How many variables are in this space? ")
    mySpace.numeqs  = getInt("How many defining equations does your space have? ")


    print("Defining ambient space...")
    Poly1=PolynomialRing(QQ,"x",mySpace.numvars)
    print Poly1
    Poly1.inject_variables()
    mySpace.setEquations()

    debug(mySpace.toString())



    print('Using the variables above, input the expression for your first equation and press return.')
    mySpace.firstequation=SR(raw_input())
    f=[]
    f.append(mySpace.firstequation)

    for i in xrange(1, mySpace.numeqs):
        print('Using the variables above, input the expression for your next equation and press return.')
        mySpace.nextequation=SR(raw_input())
        f.append(mySpace.nextequation)


    print('Check that your list of expressions is correct:')
    print f



    mySpace.fatvars = getInt("How many variables are in this fat point? ")
    mySpace.fateqs  = getInt("How many defining equations does your fat point have? ")

    print("Defining ambient space...")
    Poly2=PolynomialRing(QQ,"y",mySpace.fatvars)
    print Poly2
    Poly2.inject_variables()
    mySpace.setFatEquations()

    debug(mySpace.toString())

    print('Using the variables above, input the expression for your first Equation of your Fat point and press return.')
    mySpace.firstfatequation=SR(raw_input())
    g=[]
    g.append(mySpace.firstfatequation)

    for i in xrange(1, mySpace.fateqs):
        print('Using the variables above, input the expression for your next Equation of your Fat point and press return.')
        mySpace.nextfatequation=SR(raw_input())
        g.append(mySpace.nextfatequation)
        I=ideal(g)

    debug(mySpace.toFatString())

    
 ###################################################
 #This code computes a basis for the coordinate ring of the
 #fat point as a vector space over the rationals
 #
 ###################################################
 
    SingPoly2=singular(Poly2)
    singular.setring(SingPoly2)
    G=[str(g[i]) for i in xrange(mySpace.fateqs)]
    J=singular.ideal(G)
    J=J.groebner()
    B=list(J.kbase())
    length=len(B)
    C=[B[i].sage() for i in xrange(length)]
    
    arcvars=length*mySpace.numvars
    debug("Defining ambient space for your arc space...")


####################################################
#This block of code defines an ambient space for the arc space
#and defines the general symbolic arcs
#
####################################################

    arcvars=length*mySpace.numvars
    hh=mySpace.numvars+mySpace.fatvars+arcvars
    Poly3=PolynomialRing(QQ,"a",hh)
    Poly3.inject_variables()
    LL=list(Poly3.gens())
    LL1 = [LL[i] for i in xrange(mySpace.numvars)]
    LL2 = [LL[i] for i in xrange(mySpace.numvars,mySpace.numvars+mySpace.fatvars)]
    LL3 = [LL[i] for i in xrange(mySpace.numvars+mySpace.fatvars,hh)]
    w=Poly2.gens()
##Substitution of variables to force computation that the equations for 
##the scheme and fat point  take place in ambient space
    Dict2={w[i]:LL2[i] for i in xrange(mySpace.fatvars)}
    E=[C[i].subs(Dict2) for i in xrange(length)]
    v=Poly1.gens()
    Dict1={v[i]:LL[i] for i in xrange(mySpace.numvars)}
    F=[f[i].subs(Dict1) for i in xrange(mySpace.numeqs)]

    M=matrix(length,mySpace.numvars,LL3) 
    N=matrix(1,length, E)
##Use matrix multiplication to create the general symbolic arcs:    
    D=N*M

    DD=D.list()
    Dict2={LL1[i]:DD[i] for i in xrange(mySpace.numvars)}
    FF=[F[i].subs(Dict2) for i in xrange(mySpace.numeqs)]
    idealF=ideal(FF)
    debug(idealF)

    tempJ=list(J)
    lll=len(tempJ)
    JJ=[tempJ[i].sage() for i in xrange(lll)]
    w=Poly2.gens()
    Dict2={w[i]:LL2[i] for i in xrange(mySpace.fatvars)}
    tempI= [JJ[i].subs(Dict2) for i in xrange(lll)]
    II=ideal(tempI)
    debug(II)
##Need the following ring map in order to simplify the equations of the arc space
    QR=QuotientRing(Poly3,II)
    QR.inject_variables()
    pi=QR.cover()
##Simplification:    
    p=[PolynomialConverter(FF[i],base_ring=QQ) for i in  xrange(mySpace.numeqs)]
    rr=[p[i].symbol(FF[i]) for i in  xrange(mySpace.numeqs)]
    RR=[pi(rr[i]) for i in xrange(mySpace.numeqs)]
    debug("going to factor ring")
    d=[RR[i].lift() for i in xrange(mySpace.numeqs)]
    debug("lifting to the cover")
    
##The main algorithm. It finds the equations determined by 
##the coefficients of the basis elements.
 
    debug("Computing tempL")
    tempL=[]
    for i in  xrange(mySpace.numeqs):
        j=0
        for j in xrange(length-1):
            cc=d[i].quo_rem(E[j])
            #debug("CC: " + str(cc))
            CC=list(cc)
            tempL=tempL+[CC[0]]
            a=simplify(d[i]-CC[0]*E[j])
            if ( d[i] == a ):
                debug("No change")
            #del d[i]
            #debug("d[i] prior to change: " + str(d[i]))
            d[i] = a
            #debug("d[i] after change: " + str(d[i]))
            #d.insert(i,a)
            j=j+1
    bigL=tempL+d

##Simplify again:
    debug("... processing ...")
    quoL=[pi(bigL[i]) for i in xrange(len(bigL))]
    newL=[quoL[i].lift() for i in xrange(len(bigL))]
##This is not needed but could be useful in the future:
    #runL=[factor(newL[i]) for i in xrange(len(newL))]


##Making sure our list of equations is fully populated:
    ## What is tryL??
    breadth = int(mySpace.numeqs)
    depth   = int(length)

    tryL = []
    ## Initialize the list to -1
    for i in  xrange(breadth):
        j=0
        for j in xrange(depth):
            tryL.append("NaN")

    debug("... performing division ...")
    ## Populate list with real data
    for i in xrange(breadth):
        j=0
        for j in xrange(depth):
            idx = (i * depth + j)
            tryL[idx] = list( newL[idx].quo_rem( E[j] ))[0]
            
##Following lists are not needed but could be useful in the future:
    #tryL=[list(newL[i].quo_rem(E[i]))[0] for i in xrange(len(newL))]
    #finL=[factor(tryL[i]) for i in xrange(len(bigL))]

##Display the length of the fat point   
    debug(">>The length of your fat point is:")
    debug(length)
##Display the list of generators for the ideal which defines the arc space:   
    debug("Create ideal...")
    tempIdeal=Poly3.ideal(LL1+LL2+newL)
    debug(tempIdeal)
    
    
##The following code is an alternate display. 
##Singular has a much nicer output possible. However, for large spaces, the program hangs when creating a quotient ring in sage. 
##So, I will comment out this region, but it could be useful in the future...
#    
#    debug("Quotient ring")
#    finQR=Poly3.quotient_ring(tempIdeal)
#    finQR.inject_variables()
#  
#    
#    
#    debug( ">> Equations for Arc space: " )
#    
#    debug("Singular")
#    SingfinQR=singular(finQR)
\end{lstlisting}

\bigskip

\bigskip

\noindent\address{Andrew R. Stout\\
 Graduate Center, \\
City University of New York,\\
 365 Fifth Avenue, 10016.}


\begin{thebibliography}{1}

\bibitem[DGPS]{DGPS}
W. Decker, G.-M. Greuel, G. Pfister, H. Sch{\"o}nemann. 
\newblock {\sc Singular} {3-0-4} --- {A} computer algebra system for polynomial computations.
\newblock http://www.singular.uni-kl.de (2012).


\bibitem[DL1]{DL1} J. Denef \& F. Loeser.
\newblock  Germs of arcs of singular varieties and
motivic integration.
\newblock In {\em Inven. Math. 135}, pages 201-232, Springer, 1999.

\bibitem[DL2]{DL2} J. Denef \& F. Loeser.
\newblock Motivic Igusa Zeta Functions.
\newblock In {\em J.
Algebraic Geometry 7  no. 3}, pages 505 - 537.
1998.

\bibitem[G1]{G2} A. Grothendieck.
\newblock   \'{E}tude locale des sch\'{e}mas et des morphismes de
sch\'{e}mas. I. 
\newblock In {\em\'{E}l\'{e}ments de g\'{e}om\'{e}trie
alg\'{e}brique. IV}, Inst. Hautes \'{E}tudes Sci. Publ. Math. No. 20 1964, 259 pp. 

\bibitem[G2]{G} A. Grothendieck.
\newblock \'{E}tude locale des sch\'{e}mas et des morphismes de
sch\'{e}mas. II.
\newblock In {\em \'{E}l\'{e}ments de g\'{e}om\'{e}trie
alg\'{e}brique. IV.}, Inst. Hautes \'{E}tudes Sci. Publ. Math. No. 24 1965, 231 pp.

\bibitem[H1]{Ha1} R. Hartshorne.
\newblock  Algebraic Geometry.
\newblock Graduate Texts in
Mathematics. Springer, 1977, 496 pp.

\bibitem[H2]{Ha2} R. Hartshorne. 
\newblock Deformation Theory.
\newblock Graduate Texts in
Mathematics. Springer 2010, 236 pp. 
 

\bibitem[Liu]{Liu} Q. Liu.
\newblock Algebraic Geometry and Arithmetic Curves.
\newblock Oxford
Graduate Texts in Mathematics, Oxford University Press, 2006, 600 pp.

\bibitem[LS]{LS} Q. Liu \& J. Sebag.
\newblock The Grothendieck ring of varieties and piecewise isomorphisms.
\newblock Mathematische Zeitschrift, Volume 265, Issue 2, pp 321-342, 2010.

\bibitem[M]{M} D. Maker.
\newblock Model Theory: An Introduction.
\newblock Graduate Texts in Mathematics. Springer 2002, 345 pp.

\bibitem[NP]{NP} T. Oliphant et al.
\newblock NumPy: Open Source Scientific Tools for Python. Version 1.7.
\newblock Available at http://www.numpy.org/, 2014.


\bibitem[P]{P} Python Software Foundation. 
\newblock Python Language Reference, version 2.7. 
\newblock Available at http://www.python.org, 2014.







\bibitem[Sch1]{Sch1} H. Schoutens.
 \newblock Schemic Grothendieck Rings I. 
\newblock preprint available at {\it websupport1.citytech.cuny.edu/faculty/hschoutens/PDF/SchemicGrothendieckRingPartI.pdf} 2014.

\bibitem[Sch2]{Sch2} H. Schoutens.
\newblock  Schemic Grothendieck Rings II. 
\newblock preprint available at {\it websupport1.citytech.cuny.edu/faculty/hschoutens/PDF/SchemicGrothendieckRingPartII.pdf} 2014.

\bibitem[S]{S} W. Stein et al.
\newblock Sage Mathematics Software (Version 6.2.Beta1),
\newblock The Sage Development Team, 2014, http://www.sagemath.org.

\bibitem[Ve]{Ve} W. Veys.
\newblock Arc spaces, motivic integration, and stringy invariants.
\newblock In S. Izumiya et al. (eds.), {\em Proceedings of singularity theory and its applications}, pages 529-572. Volume 43 of {\em Advanced Studies in Pure Mathematics.} Mathematical Society of Japan, Tokyo, 2006.
\end{thebibliography}
\end{document}